\documentclass[a4paper]{amsart}
\usepackage{amsfonts}

\theoremstyle{plain}
\newtheorem{theorem}{Theorem}
\newtheorem{lemma}{Lemma}
\newtheorem{corollary}{Corollary}
\newtheorem{proposition}{Proposition}

\theoremstyle{definition}

\theoremstyle{remark}
\newtheorem{remark}{Remark}

\numberwithin{equation}{section}
\input{tcilatex}

\begin{document}
\title[A-W polynomials]{Befriending Askey--Wilson polynomials}
\author{Pawe\l\ J. Szab\l owski}
\address{Department of Mathematics and Information Sciences,\\
Warsaw University of Technology\\
ul. Koszykowa 75, 00-662 Warsaw, Poland}
\email{pawel.szablowski@gmail.com}
\date{September 1, 2011}
\thanks{The author is enormously grateful to the unknown referee who
patiently and scrupulously in two very long and detailed reports pointed out
numerous misprints, typos, small inaccuracies and mistakes thus
substantially helping to improve the paper.}

\begin{abstract}
We recall five families of polynomials constituting a part of the so-called
Askey--Wilson scheme. We do this to expose properties of the Askey--Wilson
(AW) polynomials that constitute the last, most complicated element of this
scheme. In doing so we express AW density as a product of the density that
makes $q-$Hermite polynomials orthogonal times a product of four
characteristic function of $q-$Hermite polynomials (\ref{fAW}) just pawing
the way to a generalization of AW integral. Our main results concentrate
mostly on the complex parameters case forming conjugate pairs. We present
new fascinating symmetries between the variables and some newly defined (by
the appropriate conjugate pair) parameters. In particular in (\ref%
{rozwiniecie1}) we generalize substantially famous Poisson-Mehler expansion
formula (\ref{PM}) in which $q-$Hermite polynomials are replaced by
Al-Salam--Chihara polynomials. Further we express Askey--Wilson polynomials
as linear combinations of Al-Salam--Chihara (ASC) polynomials. As a
by-product we get useful identities involving ASC polynomials. Finally by
certain re-scaling of variables and parameters we reach AW polynomials and
AW densities that have clear probabilistic interpretation.
\end{abstract}

\maketitle

\section{Introduction}

As it is well known AW polynomials were introduced in \cite{AW85} and
constitute the biggest (as of now) family of orthogonal polynomials
depending on $5$ parameters. One of these parameters is special, called base
and denoted by $q.$ In most of the applications where AW polynomials appear $%
q$ is a real from $(-1,1).$ Sometimes one considers also the case $%
q\allowbreak =\allowbreak 1$ (mostly as limiting case in probabilistic
applications) and $q>1.$ These polynomials are important in many
applications such as special functions theory (e.g. \cite{IA}),
combinatorics (e.g. \cite{Corteel10}), both non-commutative (e.g. \cite{Bo})
and classical probability (e.g. \cite{Bryc2001S} \cite{Bryc-wes10}) or last
but not least quantum mechanics (e.g. \cite{Flo97}, \cite{Floreanini97}).
Although diminished by the specialists on special functions theory as
trivial or not interesting the case $q=0$ is very important in
non-commutative probability in particular in a newly emerging so called
'free probability' to mention only \cite{Bo} and later works of R. Speicher
or D. Voiculescu and their colleagues and associates.

We present a sequence of families of orthogonal polynomials (elements of the
so called AW scheme). In this sequence the AW polynomials constitute the
last element (i.e. depending on the biggest number of parameters). By
considering this sequence we are able to notice some regularities that if
properly extended might lead to a generalization of AW polynomials. To see
these regularities we recall $5$ families of polynomials that belong to the
so-called Askey--Wilson scheme of polynomials. We present some believed to
be new, finite and infinite expansions of one family of the polynomials with
respect to the other as well as expansions of the ratio of the densities of
measures (Radon--Nikodym derivatives in fact) that make these families
orthogonal. The families of polynomials that we are going to recall are
related by the fact that respectively $4$ (leading to continuous $q-$Hermite
(qH)), $3$ (leading to big continuous $q-$Hermite (bqH)), $2$ (leading to
Al-Salam--Chihara (ASC)), $1$ (leading to continuous dual Hahn (C2H)), and
finally $0$ (leading to Askey--Wilson polynomials (AW)) of $4$ parameters
that appear in the definition of Askey--Wilson polynomials (apart of the $5$%
th parameter - the base) are set to zero.

Much is known about these polynomials however far from all. In particular
only recently moments of the AW distribution that makes these polynomials
orthogonal ware calculated (see e.g. \cite{Corteel10} or \cite{Szabl-JFA}).

The paper aims to increase this knowledge. In particular we derive the
so-called 'connection coefficients' between considered families of
polynomials and try to derive as much as possible information about these
polynomials from these coefficients. The connection coefficients are
formally known (see e.g. formulae in \cite{AW85} and \cite{IA}). In this
paper we identify these coefficients in terms of known polynomials of
parameters obtaining useful formulae especially in the case of complex
parameters.

In particular knowing the connection coefficients we are able to relate
densities (if they exist) of measures that make these polynomials
orthogonal. The idea how to get it was presented in \cite{Szabl-exp}. We
will recall this idea now briefly since below we use it to quickly and
intuitively derive the densities that make involved polynomials orthogonal.

More precisely let us assume that we have two positive probability measures $%
d\alpha $ and $d\beta $ such that $d\beta <<d\alpha $ and let respectively $%
\left\{ a_{n}\left( x\right) \right\} _{n\geq 0}$ and $\left\{ b_{n}\left(
x\right) \right\} _{n\geq 0}$ denote sets of monic polynomials orthogonal
with respect to $d\alpha $ and $d\beta .$

Suppose that we know the numbers $\gamma _{k,n}$ ('connection coefficients')
such that for every $n\geq 0:$%
\begin{equation*}
a_{n}\left( x\right) \allowbreak =\allowbreak \sum_{k=0}^{n}\gamma
_{k,n}b_{k}\left( x\right)
\end{equation*}%
for every $x\in \mathbb{R}$. Let $\frac{d\beta }{d\alpha }$ denote
Radon--Nikodym derivative of measures $d\beta $ and $d\alpha .$ Assume also
that $\int_{\limfunc{supp}\left( \beta \right) }(\frac{d\beta }{d\alpha }%
(x))^{2}d\alpha (x)<\infty ,$ then :%
\begin{equation}
\frac{d\beta }{d\alpha }\left( x\right) \allowbreak =\allowbreak
\sum_{n=0}^{\infty }\frac{\gamma _{0,n}\hat{b}_{0}}{\hat{a}_{n}}a_{n}\left(
x\right) ,  \label{series}
\end{equation}%
where $\hat{a}_{n}\allowbreak =\allowbreak \int_{\limfunc{supp}\left( \alpha
\right) }a_{n}\left( x\right) ^{2}d\alpha \left( x\right) ,$ similarly for
the polynomials $\left\{ b_{n}\right\} .$ If $d\beta $ and $d\alpha $ have
densities respectively $B(x)$ and $A(x)$ then formula (\ref{series}) may be
written as 
\begin{equation}
B(x)=A(x)\sum_{n=0}^{\infty }\frac{\gamma _{0,n}\hat{b}_{0}}{\hat{a}_{n}}%
a_{n}\left( x\right) .  \label{series1}
\end{equation}%
Convergence in (\ref{series}) is $L^{2}$ convergence, however if
coefficients $\frac{\gamma _{0,n}\hat{b}_{0}}{\hat{a}_{n}}$ are such that $%
\sum_{n\geq 0}\left( \frac{\gamma _{0,n}\hat{b}_{0}}{\hat{a}_{n}}\right)
^{2}\log ^{2}n$\allowbreak $<\allowbreak \infty ,$ then by the
Rademacher--Menshov theorem (see e.g. \cite{Alexits}) we have almost
pointwise, absolute convergence. In most cases interesting in the $q-$series
theory this condition is trivially satisfied hence all expansions we are
going to consider will be absolutely almost pointwise convergent.

Our main results concentrate on the complex parameters case and are
presented in Proposition \ref{conversion} and Theorem \ref{main}. Then the
considered families of polynomials are orthogonal with respect to the
distributions that have densities. For such parameters after additional
re-scaling and re-normalization we get clear probability interpretation of
the analyzed densities and polynomials. The interpretation is important for
both the commutative (i.e. the classical) as well as in the non-commutative
probability case. In particular in the relatively new and intensively
developing so-called 'free probability'.

As a by-product we get nontrivial finite identities involving rational
functions of $4$ variables (see Corollary \ref{identies}).

It turns out that in this 'complex parameter' case there appear new
fascinating symmetries between the variables and some of the parameters. In
particular we express Askey--Wilson polynomials as linear combinations of
Al-Salam--Chihara polynomials which together with obtained earlier expansion
of the Askey--Wilson density forms complete generalization of the situation
met in the case of Al-Salam--Chihara and $q-$Hermite polynomials and the
Poisson--Mehler expansion formula (see Corollary \ref{AWexp} and Remark \ref%
{generalization}).

The paper is organized as follows. In the next Section \ref{intr} we present
elements of the useful notation used in the so called $q-$series theory,
define AW polynomials and relate them to other families of orthogonal
polynomials that are the object of this paper and point out relationships
between them that will be important in the sequel. The results in this
section are rather easy if not trivial and we present them and sketch their
proofs for the sake of completeness of the paper.

In Section \ref{complex} we present our main results. Starting with the
useful conversion proposition (Proposition \ref{conversion}) that allows to
express a certain combination of complex exponentials in the form of a
certain combination of $q-$Hermite and Al-Salam--Chihara polynomials, we
present our main result (Theorem \ref{main}) that is the finite expansion of 
$n$-th Askey--Wilson polynomial with complex parameters as a linear
combination of ASC polynomials with coefficients that have a form of some
polynomial closely related to ASC polynomials and conversely. That is we
present $n-$th ASC polynomial as a linear combination of AW polynomials with
coefficients being ASC polynomials of certain parameters. In Section \ref%
{prob} we give probabilistic interpretation of the AW polynomials as well as
the AW density. We also present the exact form of these polynomials and AW
density for the special cases $q\allowbreak =\allowbreak 1$ and $%
q\allowbreak =\allowbreak 0.$ The last case being important for the free
probability. Last Section \ref{dowody} presents less interesting or longer
proofs.

\section{Notation and definitions of families of orthogonal polynomials\label%
{intr}}

\subsection{Notation}

To help readers not familiar with notation used in the $q-$series theory we
present below a brief review of the notation used in it.

$q$ is a parameter. We will assume that $-1<q\leq 1$ unless otherwise stated.

We set $\left[ 0\right] _{q}\allowbreak =\allowbreak 0,$ $\left[ n\right]
_{q}\allowbreak =\allowbreak 1+q+\ldots +q^{n-1}\allowbreak ,$ $\left[ n%
\right] _{q}!\allowbreak =\allowbreak \prod_{j=1}^{n}\left[ j\right] _{q},$
with $\left[ 0\right] _{q}!\allowbreak =1,$%
\begin{equation*}
\QATOPD[ ] {n}{k}_{q}\allowbreak =\allowbreak \left\{ 
\begin{array}{ccc}
\frac{\left[ n\right] _{q}!}{\left[ n-k\right] _{q}!\left[ k\right] _{q}!} & 
, & n\geq k\geq 0, \\ 
0 & , & otherwise.%
\end{array}%
\right.
\end{equation*}%
We will use also the so called $q-$Pochhammer symbol for $n\geq 1:$%
\begin{equation*}
\left( a;q\right) _{n}=\prod_{j=0}^{n-1}\left( 1-aq^{j}\right) ,~~\left(
a_{1},a_{2},\ldots ,a_{k};q\right) _{n}\allowbreak =\allowbreak
\prod_{j=1}^{k}\left( a_{j};q\right) _{n},
\end{equation*}%
with $\left( a;q\right) _{0}=1$.

Often $\left( a;q\right) _{n}$ as well as $\left( a_{1},a_{2},\ldots
,a_{k};q\right) _{n}$ will be abbreviated to $\left( a\right) _{n}$ and 
\newline
$\left( a_{1},a_{2},\ldots ,a_{k}\right) _{n},$ if the base will be $q$ and
if such abbreviation will not cause misunderstanding.

We will need also the following four well known, easy to derive, formulae
that are valid for $q\neq 0$ and $a\neq 0:$%
\begin{eqnarray}
\left( a\right) _{n}\allowbreak &=&\allowbreak \left( -1\right) ^{n}q^{%
\binom{n}{2}}a^{n}\left( a^{-1};q^{-1}\right) _{n},~\left( a;q^{-1}\right)
_{n}=\left( -1\right) ^{n}q^{-\binom{n}{2}}a^{n}\left( 1/a\right) _{n},
\label{q_q^-1} \\
\left( a\right) _{n} &=&\left( aq^{n-1};q^{-1}\right) _{n},~\QATOPD[ ] {n}{k}%
_{q^{-1}}=\QATOPD[ ] {n}{k}_{q}q^{k(k-n)}.  \label{qbinq-1}
\end{eqnarray}

It is easy to notice that for $q\allowbreak \in \allowbreak (-1,1)$ we have: 
$\left( q\right) _{n}=\left( 1-q\right) ^{n}\left[ n\right] _{q}!$ and that 
\begin{equation*}
\QATOPD[ ] {n}{k}_{q}\allowbreak =\allowbreak \allowbreak \left\{ 
\begin{array}{ccc}
\frac{\left( q\right) _{n}}{\left( q\right) _{n-k}\left( q\right) _{k}} & ,
& n\geq k\geq 0, \\ 
0 & , & otherwise.%
\end{array}%
\right.
\end{equation*}%
\newline
To support intuition let us notice that 
\begin{equation*}
\left[ n\right] _{1}\allowbreak =\allowbreak n,\left[ n\right]
_{1}!\allowbreak =\allowbreak n!,~~\QATOPD[ ] {n}{k}_{1}\allowbreak
=\allowbreak \binom{n}{k},~~\left( a;1\right) _{n}\allowbreak =\allowbreak
\left( 1-a\right) ^{n}
\end{equation*}%
and 
\begin{eqnarray*}
\left[ n\right] _{0}\allowbreak &=&\allowbreak \left\{ 
\begin{array}{ccc}
1 & if & n\geq 1, \\ 
0 & if & n=0,%
\end{array}%
\right. \left[ n\right] _{0}!\allowbreak =\allowbreak 1,\QATOPD[ ] {n}{k}%
_{0}\allowbreak =\allowbreak 1, \\
\left( a;0\right) _{n}\allowbreak &=&\allowbreak \left\{ 
\begin{array}{ccc}
1 & if & n=0, \\ 
1-a & if & n\geq 1.%
\end{array}%
\right.
\end{eqnarray*}

\subsection{Askey--Wilson and other related families of polynomials}

These families of orthogonal polynomials form a part of the so called
Askey--Wilson scheme of orthogonal polynomials. It means that they are
interrelated and each member of this sequence depends on one more parameter.
All of the families that will be presented below depend on one common
parameter - the base denoted by $q$ such that $q\allowbreak >\allowbreak -1.$
For $q\in (-1,1)$ these polynomials are orthogonal with respect to some
positive measure on $[-1,1]$ under some conditions imposed on the values of
parameters. Hence we will mostly concentrate on the case $q\in (-1,1).$ The
conditions imposed on the parameters consist in fact of the requirement that
the products of all pairs of these parameters are real and have absolute
values not exceeding $1.$ If one of the parameters, say $a$ is real and
greater than $1$ then one has to assume that $q\neq 0$ and then the measure
has $\#\left\{ k:aq^{k}>1\right\} $ jumps located at points $%
x_{k}\allowbreak =\allowbreak \left( (aq^{k}+(aq^{k}\right) ^{-1})/2.$ The
masses that are assigned to these points depend on the polynomial. Since our
main concern is with the regular, 'without jumps' case we will not give
those masses.

Askey--Wilson (AW) polynomials $\left\{ \alpha _{n}\right\} $ were
introduced in \cite{AW85}. Its original definition was given through the
basic hypergeometric function. Fortunately there exist other
characterizations of this family of polynomials. Particularly there exists a
characterization through three-term recurrence relation . Namely let $%
p_{n}\left( x|a,b,c,d,q\right) $ denote AW polynomial as presented in (\cite%
{KLS}, (14.1.1)) and let us define 
\begin{equation}
\alpha _{n}\left( x|a,b,c,d,q\right) \allowbreak =\allowbreak p_{n}\left(
x|a,b,c,d,q\right) /\left( abcdq^{n-1}\right) _{n}.  \label{correction}
\end{equation}%
Then we get the following three-term recurrence relation to be satisfied by
the polynomials $\alpha _{n}$:%
\begin{equation}
2x\alpha _{n}\left( x\right) =\alpha _{n+1}\left( x\right) +e_{n}\left(
a,b,c,d,q\right) \alpha _{n}(x)-f_{n}\left( a,b,c,d,q\right) \alpha
_{n-1}(x),  \label{_aw}
\end{equation}%
with $\alpha _{-1}\left( x\right) \allowbreak =\allowbreak 0,$ $\alpha
_{0}\left( x\right) \allowbreak =\allowbreak 1,$ where for simplicity we
denoted $\alpha _{n}(x)\allowbreak =$\newline
$\allowbreak \alpha _{n}\left( x|a,b,c,d,q\right) $ and $f_{n}$ and $e_{n}$
are given below as assertion of the Proposition \ref{wsp}.

We have immediate

\begin{proposition}
\label{wsp}%
\begin{gather*}
f_{n}(a,b,c,d,q)=(1-q^{n})\times \\
\frac{\left(
abq^{n-1},acq^{n-1},adq^{n-1},bcq^{n-1},bdq^{n-1},cdq^{n-1},abcdq^{n-2}%
\right) _{1}}{(abcdq^{2n-3})_{3}(abcdq^{2n-2})_{1}}, \\
e_{n}(a,b,c,d,q)=\frac{q^{n-2}}{(1-abcdq^{2n-2})\left( 1-abcdq^{2n}\right) }%
\times \\
((a+b+c+d)(q^{2}-abcdq^{n}(1+q-q^{n+1}))+ \\
(abc+abd+acd+bcd)(q-q^{n+2}-q^{n+1}+abcdq^{2n})).
\end{gather*}
\end{proposition}

\begin{proof}
Technical, uninteresting proof is moved to Section \ref{dowody}. It is based
on formula (\cite{KLS}, (14.1.4)).
\end{proof}

From Favard's theorem it follows that polynomials $\left\{ \alpha
_{n}\right\} $ are orthogonal with respect to the positive measure provided $%
f_{n}\left( a,b,c,d,q\right) $ is nonnegative for all $n\geq 1.$ Examining
the form of $f_{n}$ and $e_{n}$ we see that this requirement can be reduced
to the requirement that $\left\vert ab\right\vert ,$ $\left\vert
ac\right\vert ,$ $\left\vert ad\right\vert ,$ $\left\vert bc\right\vert ,$ $%
\left\vert bd\right\vert ,$ $\left\vert cd\right\vert \allowbreak \leq
\allowbreak 1.$ If $\left\vert a\right\vert ,$ $\left\vert b\right\vert ,$ $%
\left\vert c\right\vert ,$ $\left\vert d\right\vert \leq 1$ then this
measure has density.

Let us remark that AW polynomials with some of the parameters $a,b,c,d$ set
to zero constitute new families of polynomials and have separate names.
Namely 
\begin{eqnarray*}
\alpha _{n}(x|0,b,c,d,q)\allowbreak &=&\allowbreak \psi
_{n}(x|b,c,d,q),~\alpha _{n}(x|0,0,c,d,q)=Q_{n}(x|c,d,q), \\
\alpha _{n}(x|0,0,c,d,q) &=&h_{n}(x|d,q),~\alpha
_{n}(x|0,0,0,0,q)=h_{n}(x|q),
\end{eqnarray*}%
and polynomials $\left\{ \psi _{n}(x|b,c,d,q)\right\} ,$ $\left\{
Q_{n}(x|c,d,q)\right\} ,$ $\left\{ h_{n}(x|d,q)\right\} ,$ $\left\{
h_{n}(x|q)\right\} $ are called in the literature (see e.g. \cite{IA} or 
\cite{KLS}) respectively continuous dual Hahn, Al-Salam--Chihara, continuous
big $q-$Hermite and continuous $q-$Hermite.

These polynomials are related to one another by the relationships similar to
(\ref{a_na_c}) and (\ref{c_na_a}) (below) with appropriate parameters set to
zero.

\begin{lemma}
\label{conAW-c2h}We have 
\begin{gather}
\alpha _{n}(x|a,b,c,d,q)\allowbreak =\allowbreak \sum_{i=0}^{n}\QATOPD[ ] {n%
}{i}_{q}\left( -a\right) ^{n-i}q^{\binom{n-i}{2}}\frac{\left(
bcq^{i},bdq^{i},cdq^{i}\right) _{n-i}}{\left( abcdq^{n+i-1}\right) _{n-i}}%
\psi _{i}\left( x|b,c,d,q\right) ,  \label{a_na_c} \\
\psi _{n}\left( x|b,c,d,q\right) =\sum_{i=0}^{n}\QATOPD[ ] {n}{i}_{q}a^{n-i}%
\frac{\left( bcq^{i},bdq^{i},cdq^{i}\right) _{n-i}}{\left( abcdq^{2i}\right)
_{n-i}}\alpha _{i}\left( x|a,b,c,d,q\right) .  \label{c_na_a}
\end{gather}
\end{lemma}

\begin{proof}
Strongly basing on formula (\cite{AW85}, (6.4)) is shifted to Section \ref%
{dowody}.
\end{proof}

\begin{remark}
In deriving (\ref{a_na_c}) and (\ref{c_na_a}) we used (\cite{AW85}, (6.4)).
Note that there exits also slightly different formula (\cite{IA}, (16.4.3)).
\end{remark}

Let us introduce two functions whose definitions were taken from \cite{KLS}.
They have the following interpretation. One of them is a generating function
of $q-$Hermite polynomials i.e.

\begin{equation}
\varphi _{h}\left( x|\rho ,q\right) =\sum_{n\geq 0}\frac{\rho ^{n}}{\left(
q\right) _{n}}h_{n}\left( x|q\right) =\frac{1}{\prod_{i=0}^{\infty }v\left(
x|\rho q^{i}\right) },\text{ with }v\left( x|t\right) \allowbreak
=\allowbreak 1-2xt+t^{2}  \label{fi_h}
\end{equation}%
convergent for $\left\vert \rho \right\vert <1,$ $\left\vert x\right\vert
\leq 1.$ (\cite{KLS}, (14.26.11)). Let us observe that for all $x\in \lbrack
-1,1],$ $t\in \mathbb{R}:v\left( x|t\right) \geq 0.$

The other one is the density with respect to which $q-$ Hermite polynomials
are orthogonal i.e. we have:%
\begin{equation*}
\int_{\lbrack -1,1]}h_{n}\left( x|q\right) h_{m}\left( x|q\right)
f_{h}\left( x|q\right) dx=\delta _{m,n}\left( q\right) _{n},
\end{equation*}%
where%
\begin{equation}
f_{h}\left( x|q\right) \allowbreak =\allowbreak \frac{2\left( q\right)
_{\infty }\sqrt{1-x^{2}}}{\pi }\prod_{i=1}^{\infty }l\left( x|q^{i}\right) ,
\label{ort_h}
\end{equation}%
with $l\left( x|a\right) \allowbreak =\allowbreak (1+a)^{2}-4ax^{2}$ (see 
\cite{KLS}, (14.26.2)). Notice that for all $\left\vert x\right\vert \leq
1:l\left( x|a\right) \geq 0.$

Using these functions one can express densities of the measures that make
other families of polynomials orthogonal. Namely it is an easy observation
that formulae respectively (14.1.2), (14.3.2), (14.8.2), (14.18.2) of \cite%
{KLS} can be written in form for $\left\vert a\right\vert ,$ $\left\vert
b\right\vert ,$ $\left\vert c\right\vert ,$ $\left\vert d\right\vert <1$ 
\begin{gather}
f_{AW}\left( x|a,b,c,d,q\right) =f_{h}\left( x|q\right) \varphi _{h}\left(
x|a,q\right) \varphi _{h}\left( x|b,q\right) \varphi _{h}\left( x|c,q\right)
\varphi _{h}\left( x|d,q\right) \times  \label{fAW} \\
\frac{\left( ab,ac,ad,bc,bd,cd\right) _{\infty }}{\left( abcd\right)
_{\infty }},  \notag \\
f_{\psi }\left( x|a,b,c,q\right) \allowbreak =\allowbreak f_{h}\left(
x|q\right) \varphi _{h}\left( x|a,q\right) \varphi _{h}\left( x|b,q\right)
\varphi _{h}\left( x|c,q\right) \left( ab,ac,bc\right) _{\infty },
\label{gest-psi} \\
f_{Q}\left( x|a,b,q\right) =f_{h}\left( x|q\right) \varphi _{h}\left(
x|a,q\right) \varphi _{h}\left( x|b,q\right) \left( ab\right) _{\infty },
\label{gest_Q} \\
f_{bH}\left( x|a,q\right) =\allowbreak f_{h}\left( x|q\right) \varphi
_{h}\left( x|a,q\right) .  \label{f_bh}
\end{gather}
These functions are the densities of measures that make respectively AW,
C2H, ASC, bqH polynomials orthogonal. Let us remark that we can get these
formulae expressing densities using the presented above idea of the ratio of
the density expansion and formulae (\ref{c_na_a}) and (\ref{a_na_c}) with
appropriate parameters set to zero as 'suppliers' of connection coefficients$%
.$ We can also follow the path indicated in the Remark below and get (\ref%
{fAW}) and then by assigning zero values to appropriate parameters get the
remaining densities.

Let us remark that the above mentioned expressions for the densities $%
f_{AW}, $ $f_{\psi },$ $f_{Q},$ $f_{bH}$ can be the base for a
generalization of Askey--Wilson integral that is the statement that $\int
f_{AW}\allowbreak =\allowbreak 1.$ More precisely they allow definition of
nonnegative functions of $x$ depending on more than $4$ parameters that when
reduced to respectively $4,$ $3,$ $2,$ $1$ parameters (the others set to
zero) are equal to $f_{AW},$ $f_{\psi },$ $f_{Q},$ $f_{bH}.$ For details see 
\cite{Szabl-intAW}.

\begin{remark}
Following the fact that $\int_{-1}^{1}f_{AW}\left( x|a,b,c,d,q\right)
dx\allowbreak =\allowbreak 1$ and of course (\ref{fAW}) we have:%
\begin{eqnarray*}
&&\sum_{j,k,n,m}\frac{a^{j}b^{k}c^{n}d^{m}}{\left( q\right) _{j}\left(
q\right) _{k}\left( q\right) _{n}\left( q\right) _{m}}\int_{-1}^{1}h_{j}%
\left( x|q\right) h_{k}\left( x|q\right) h_{n}\left( x|q\right) h_{m}\left(
x|q\right) f_{h}\left( x|q\right) dx \\
&=&\frac{\left( abcd\right) _{\infty }}{\left( ab,ac,ad,bc,bd,cd\right)
_{\infty }},
\end{eqnarray*}%
i.e. generating function of the numbers \newline
$\left\{ \int_{-1}^{1}h_{j}\left( x|q\right) h_{k}\left( x|q\right)
h_{n}\left( x|q\right) h_{m}\left( x|q\right) f_{h}\left( x|q\right)
dx\right\} _{j,k,n,m\geq 0}$ for free. Similarly by setting say $%
a\allowbreak =\allowbreak 0$ in the above mentioned formula we can get
generating function of the numbers $\left\{ \int_{-1}^{1}h_{j}\left(
x|q\right) h_{k}\left( x|q\right) h_{n}\left( x|q\right) f_{h}\left(
x|q\right) dx\right\} _{j,k,n\geq 0}$. \newline
Numbers $\int_{-1}^{1}h_{i}\left( x|q\right) h_{j}\left( x|q\right)
h_{k}\left( x|q\right) f_{h}\left( x|q\right) dx$ have been calculated and
are presented in many many equivalent forms.
\end{remark}

\begin{remark}
Following the idea of expansion, assuming $\left\vert a\right\vert ,$ $%
\left\vert b\right\vert ,$ $\left\vert c\right\vert ,$ $\left\vert
d\right\vert \leq 1,$ using (\ref{c_na_a}) and (\ref{alfa^2}) with $%
a\allowbreak =\allowbreak 0$ we can relate density of the Askey--Wilson
polynomials and the density of C2H polynomials. Namely we have following (%
\ref{series1}):%
\begin{equation}
f_{AW}\left( x|a,b,c,d,q\right) =f_{\psi }\left( x|b,c,d,q\right) \sum_{j}%
\frac{a^{j}}{\left( abcd,q\right) _{j}}\psi _{j}\left( x|b,c,d,q\right) .
\label{rozwAW}
\end{equation}%
Expression $\sum_{j}\frac{a^{j}}{\left( abcd,q\right) _{j}}\psi _{j}\left(
x|b,c,d,q\right) $ was summed recently by Atakishiyeva and Atakishiyev in 
\cite{At-At11} yielding 
\begin{equation*}
\sum_{j}\frac{a^{j}}{\left( abcd,q\right) _{j}}\psi _{j}\left(
x|b,c,d,q\right) =\frac{\left( ab,ac,ad\right) _{\infty }}{\left(
abcd\right) _{\infty }}\varphi _{h}\left( x|a,q\right) .
\end{equation*}%
Hence following (\ref{gest-psi}) we get (\ref{fAW}). This formula is known
for nearly 30 years of course in a different form (see \cite{AW85}, \cite{IA}%
, \cite{KLS}). The presented above one exposes symmetry of this density with
respect to the parameters. However notice that if it was assumed to be known
then we would be able by (\ref{rozwAW}) to sum $\sum_{j\geq 0}\frac{a^{j}}{%
\left( abcd,q\right) _{j}}\psi _{j}\left( x|b,c,d,q\right) $ saving
Atakishiyeva and Atakishiyev a lot of work. Whatever point of view we chose,
(\ref{rozwAW}) gives a new interesting interpretation to the generating
function $\sum_{j\geq 0}\frac{a^{j}}{\left( abcd,q\right) _{j}}\psi
_{j}\left( x|b,c,d,q\right) .$
\end{remark}

\begin{remark}
For completeness let us remark that following (\cite{IA}, (15.2.4)) or (\cite%
{KLS}, (14.1.2)) and slightly different (rescaled) definition of the density 
$f_{AW}$ (compare (\ref{fAW}) and \cite{IA}, (15.2.2)) we have:%
\begin{gather}
\int_{\lbrack -1,1]}\alpha _{n}\left( x|a,b,c,d,q\right) \alpha _{m}\left(
x|a,b,c,d,q\right) f_{AW}\left( x|a,b,c,d,q\right) dx  \label{alfa^2} \\
=\delta _{mn}\frac{\left( abcdq^{n-1}\right) _{n}\left(
ab,ac,ad,bc,bd,cd,q\right) _{n}}{\left( abcd\right) _{2n}}.  \notag
\end{gather}
\end{remark}

In the next section when discussing complex parameter case we will need the
following corollary of Lemma \ref{conAW-c2h}.

\begin{corollary}
\label{conection}%
\begin{gather}
\alpha _{n}(x|a,b,c,d,q)=\sum_{k=0}^{n}\left[ \QATOP{n}{k}\right] _{q}\left(
-1\right) ^{n-k}q^{\binom{n-k}{2}}\left( cdq^{k}\right) _{n-k}Q_{k}\left(
x|c,d,q\right)  \label{aw_na_asc} \\
\times \sum_{m=0}^{n-k}\left[ \QATOP{n-k}{m}\right]
_{q}q^{m(m-n+k)}a^{m}b^{n-k-m}\frac{\left( bcq^{n-m},bdq^{n-m}\right) _{m}}{%
\left( abcdq^{2n-m-1}\right) _{m}}.  \notag \\
Q_{n}\left( x|c,d,q\right) \allowbreak =\sum_{j=0}^{n}\left[ \QATOP{n}{j}%
\right] _{q}\left( cdq^{j}\right) _{n-j}\alpha _{j}\left( x|a,b,c,d,q\right)
\label{asw_na_aw} \\
\times \sum_{m=0}^{n-j}\left[ \QATOP{n-j}{m}\right] _{q}b^{n-j-m}a^{m}\frac{%
\left( bcq^{j},bdq^{j}\right) _{m}}{\left( abcdq^{2j}\right) _{m}}.  \notag
\end{gather}
\end{corollary}

\begin{proof}
Simple proof is shifted to Section \ref{dowody}.
\end{proof}

\section{Complex parameters\label{complex}}

As mentioned above all analyzed polynomials are orthogonal with respect to a
positive measure if only products of all pairs of their parameters have
absolute values less that $1.$ Hence these parameters can also be complex
but forming conjugate pairs.

In this section we will consider new parameters $y,$ $z,$ $\rho _{1},$ $\rho
_{2}$ related to parameters $a,$ $b,$ $c,$ $d$ by the following relationship:%
\begin{equation}
a=\rho _{1}\exp \left( i\theta \right) ,~b=\rho _{1}\exp \left( -i\theta
\right) ,~c=\rho _{2}\exp \left( i\eta \right) ,~d=\rho _{2}\exp \left(
-i\eta \right) ,  \label{newparam}
\end{equation}%
with $y\allowbreak =\allowbreak \cos \theta ,$ $z\allowbreak =\allowbreak
\cos \eta $ and $i$ denoting imaginary unit. Conditions $|ab|,$ $|ac|,$ $%
|ad|,$ $|bc|,$ $|bd|,$ $|cd|\allowbreak \leq 1$ guarantee that the measure
orthogonalizing appropriate orthogonal polynomials is positive. They are
reduced in this case to the requirement that $|\rho _{1}|$, $\left\vert \rho
_{2}\right\vert \leq 1.$ However conditions $|\rho _{1}|$, $\left\vert \rho
_{2}\right\vert \leq 1$ imply that we have also $\left\vert a\right\vert ,$ $%
\left\vert b\right\vert ,$ $\left\vert c\right\vert ,$ $\left\vert
d\right\vert $ $\leq 1$ that is the density of this measure exists.

Specially important will be in this section polynomials that depend on even
or zero number of parameters i.e. $q-$Hermite, ASC and AW polynomials, since
then the measures that make these polynomials orthogonal have densities.

First let us consider ASC polynomials with complex parameters $a\allowbreak
=\allowbreak \rho \exp \left( i\theta \right) ,$ $b\allowbreak =\allowbreak
\rho \exp \left( -i\theta \right) $.

Let us denote $Q_{n}\left( x|\rho \exp \left( i\theta \right) ,\rho \exp
\left( -i\theta \right) ,q\right) \allowbreak =\allowbreak p_{n}\left(
x|y,\rho ,q\right) ,$ with $\cos \theta \allowbreak =\allowbreak y.$

One can easily notice that polynomials $p_{n}$ satisfy the following
three-term recurrence relation:%
\begin{equation}
p_{n+1}\left( x|y,\rho ,q\right) =2(x-\rho yq^{n})p_{n}(x|y,\rho
,q)-(1-q^{n})(1-\rho ^{2}q^{n-1})p_{n-1}\left( x|y,\rho ,q\right) .
\label{_3tr_p}
\end{equation}%
Notice also that for $q\allowbreak =\allowbreak 0$ (\ref{_3tr_p}) takes the
following form for $n\geq 1:$%
\begin{equation*}
p_{n+1}\left( x|y,\rho ,0\right) \allowbreak =\allowbreak 2xp_{n}(x|y,\rho
,0)-p_{n-1}(x|y,\rho ,0),
\end{equation*}%
with $p_{0}(x|y,\rho ,0)\allowbreak =\allowbreak 1,$ $p_{1}(x|y,\rho
,0)\allowbreak =\allowbreak 2(x-\rho y).$ Comparing this with 3-term
recurrence satisfied by Chebyshev polynomials $\left\{ U_{n}\right\} $ (see
e.g. \cite{KLS}, (9.8.40)) we see that: 
\begin{eqnarray*}
p_{0}(x|y,\rho ,0) &=&U_{0}(x),~p_{1}(x|y,\rho ,0)=U_{1}(x)-\rho yU_{0}(x),
\\
p_{n}(x|y,\rho ,0)\allowbreak &=&\allowbreak U_{n}(x)-2\rho yU_{n-1}(x)+\rho
^{2}U_{n-1}(x),
\end{eqnarray*}%
for $n>1.$

Following (\cite{bms}, (1.2)) we have: 
\begin{equation*}
\varphi _{p}\left( x,t|y,\rho ,q\right) =\sum_{j=0}^{\infty }\frac{t^{j}}{%
\left( q\right) _{j}}p_{j}\left( x|y,\rho ,q\right) =\prod_{j=0}^{\infty }%
\frac{v\left( y|\rho tq^{j}\right) }{v\left( x|tq^{j}\right) }.
\end{equation*}

Notice that 
\begin{equation}
\varphi _{p}\left( x,t|y,\rho ,q\right) \allowbreak =\allowbreak \frac{%
\varphi _{h}\left( x|t,q\right) }{\varphi _{h}\left( y|t\rho ,q\right) }.
\label{chp2}
\end{equation}

Let us introduce also the following $\left\{ b_{n}\left( x|q\right) \right\}
_{n\geq -1}$ auxiliary family of polynomials defined for $q\neq 0$ by the
formula:%
\begin{equation}
b_{n}\left( x|q\right) \allowbreak =\allowbreak \left( -1\right) ^{n}q^{%
\binom{n}{2}}h_{n}\left( x|q^{-1}\right) ,  \label{b_h}
\end{equation}%
and for $q\allowbreak =\allowbreak 0$ we set $b_{0}(x|0)\allowbreak
=\allowbreak 1,$ $b_{1}(x|0)\allowbreak =\allowbreak -2x,$ $%
b_{2}(x|0)\allowbreak =\allowbreak 1$ and $b_{n}(x|0)\allowbreak
=\allowbreak 0$ for $n\allowbreak =\allowbreak -1,3,4,\ldots $ . These
polynomials were studied in \cite{Ask89}, \cite{AI84} and \cite{IsMas} and
some of their basic properties can be found there. Their r\^{o}le in the
complex parameters case was disclosed in \cite{bms}. We collect these
properties and expose the r\^{o}le of these polynomials in the following
Lemma.

\begin{lemma}
\label{zal_b_h}For all $n\geq 0,$ $x,y,\rho \in \mathbb{R},$ i )%
\begin{eqnarray*}
b_{n+1}\left( x|q\right) \allowbreak &=&\allowbreak -2q^{n}xb_{n}\left(
x|q\right) +q^{n-1}(1-q^{n})b_{n-1}\left( x|q\right) , \\
\forall q &\neq &0:b_{n}\left( x|q^{-1}\right) \allowbreak =\allowbreak
(-1)^{n}q^{-\binom{n}{2}}h_{n}\left( x|q\right) ,
\end{eqnarray*}

ii) 
\begin{equation*}
\sum_{j\geq 0}\frac{t^{j}}{\left( q\right) _{j}}b_{j}\left( x|q\right)
\allowbreak =\allowbreak \prod_{k=0}^{\infty }v\left( x|tq^{k}\right)
\allowbreak =\allowbreak 1/\varphi _{h}\left( x|t,q\right) ,
\end{equation*}%
consequently for all $n\geq 1:\sum_{j=0}^{n}\QATOPD[ ] {n}{j}_{q}h_{j}\left(
x|q\right) b_{n-j}\left( x|q\right) \allowbreak =\allowbreak 0,$

iii) 
\begin{eqnarray*}
p_{n}\left( x|y,\rho ,q\right) &=&\sum_{j=0}^{n}\QATOPD[ ] {n}{j}_{q}\rho
^{n-j}h_{j}\left( x|q\right) b_{n-j}\left( y|q\right) , \\
p_{n}\left( x|y,0,q\right) \allowbreak &=&\allowbreak h_{n}\left( x|q\right)
,
\end{eqnarray*}

iv) 
\begin{equation*}
h_{n}\left( x|q\right) \allowbreak =\allowbreak \sum_{j=0}^{n}\QATOPD[ ] {n}{%
j}_{q}p_{j}\left( x|y,\rho ,q\right) \rho ^{n-j}h_{n-j}\left( y|q\right) .
\end{equation*}
\end{lemma}

\begin{proof}
We prove first of our identities for $x,y\in \allowbreak \lbrack -1,1],$ $%
\left\vert \rho \right\vert <1.$ Then since all involved functions are
polynomials we extend them to all values of variables. i) the second
assertion of i) follows directly definition of $b_{n}.$ To get the other one
we start with three-term recurrence relation satisfied for polynomials $%
h_{n} $ i.e. $h_{n+1}\left( x|q^{-1}\right) \allowbreak =\allowbreak
2xh_{n}\left( x|q^{-1}\right) -(1-q^{-n})h_{n-1}\left( x|q^{-1}\right) .$
Then we multiply both sides by $(-1)^{n+1}q^{\binom{n+1}{2}}$ and utilize
the fact that $\binom{n}{2}+n\allowbreak =\allowbreak \binom{n+1}{2}$. ii)
and iii) with different scaling and normalizations were shown in \cite{bms}.
iv) was proved in \cite{IRS99}.
\end{proof}

Let us introduce another auxiliary family of polynomials $\left\{
g_{n}\left( x|y,\rho ,q\right) \right\} _{n\geq -1}$ that are related to
polynomials $p_{n}$ in the following way: 
\begin{equation}
g_{n}\left( x|y,\rho ,q\right) \allowbreak =\allowbreak \left\{ 
\begin{array}{ccc}
\rho ^{n}p_{n}\left( y|x,\rho ^{-1},q\right) & if & \rho \neq 0, \\ 
b_{n}\left( x|q\right) & if & \rho =0.%
\end{array}%
\right.  \label{_g}
\end{equation}%
We have the following easy Proposition.

\begin{proposition}
\label{odwr_p}i) For all $n\geq 0,$ and $q\neq 0:$%
\begin{equation*}
g_{n}\left( x|y,\rho ,q\right) \allowbreak =\allowbreak (-1)^{n}q^{\binom{n}{%
2}}p_{n}\left( x|y,\rho ,q^{-1}\right) ,
\end{equation*}%
ii) for all $n\geq 0:$%
\begin{equation*}
g_{n+1}\left( x|y,\rho ,q\right) \allowbreak =\allowbreak -2(xq^{n}-\rho
y)g_{n}\left( x|y,\rho ,q\right) -(1-q^{n})(\rho ^{2}-q^{n-1})g_{n-1}\left(
x|y,\rho ,q\right) ,
\end{equation*}%
\newline
with $g_{-1}\left( x|y,\rho ,q\right) \allowbreak =\allowbreak 0,$ $%
g_{0}\left( x|y,\rho ,q\right) \allowbreak =\allowbreak 1,$

iii) for all $\left\vert t\right\vert ,\left\vert q\right\vert
<1,~\left\vert x\right\vert ,\left\vert y\right\vert \leq 1:$%
\begin{equation*}
\sum_{j=0}^{\infty }\frac{t^{n}}{\left( q\right) _{n}}g_{n}\left( x|y,\rho
,q\right) \allowbreak =\allowbreak 1/\varphi _{p}\left( x,t|y,\rho ,q\right)
,
\end{equation*}

iv) for all $n\geq 1,x,y,\rho \in \mathbb{R}:$%
\begin{equation*}
\sum_{j=0}^{n}\QATOPD[ ] {n}{j}_{q}p_{j}\left( x|y,\rho ,q\right)
g_{n-j}\left( x|y,\rho ,q\right) =0.
\end{equation*}
\end{proposition}

\begin{proof}
Is shifted to Section \ref{dowody}.
\end{proof}

Finally let us introduce the following polynomials of order $2$ in $x$ and $%
y $ and $4$ in $\rho :$%
\begin{equation}
\omega \left( x,y|\rho \right) \allowbreak =\allowbreak (1-\rho
^{2})^{2}-4xy\rho (1+\rho ^{2})+4\rho ^{2}\left( x^{2}+y^{2}\right) .
\label{gr_w}
\end{equation}

We have obvious observations collected in the following proposition:

\begin{proposition}
\label{reduce}i) $\omega \left( x,y|\rho \right) =v\left( x|\rho e^{i\theta
}\right) v\left( x|\rho e^{-i\theta }\right) ,$ where $y\allowbreak
=\allowbreak \cos \left( \theta \right) $ and the polynomial $v$ is defined
by (\ref{fi_h}).

ii) $\left( te^{i(\theta +\eta )},te^{i(\theta -\eta )},te^{i(-\theta -\eta
)},te^{i(-\theta +\eta )}\right) _{1}\allowbreak =\allowbreak \omega \left(
y,z|t\right) ,$ with $y\allowbreak =\allowbreak \cos \left( \theta \right) ,$
$z\allowbreak =\allowbreak \cos \left( \eta \right) .$
\end{proposition}

Besides one knows also the density with respect to which modified ASC
polynomials are orthogonal. Namely in the present setting formula (\ref%
{alfa^2} with $d=c=0$) has the following form: 
\begin{eqnarray}
\int_{-1}^{1}p_{n}\left( x|y,\rho ,q\right) p_{m}\left( x|y,\rho ,q\right)
f_{p}\left( x|y,\rho ,q\right) \allowbreak &=&\allowbreak \delta _{nm}\left(
\rho ^{2},q\right) _{n},  \label{_ort} \\
~f_{p}\left( x|y,\rho ,q\right) \allowbreak &=&\allowbreak f_{h}\left(
x|q\right) \frac{\left( \rho ^{2}\right) _{\infty }}{\prod_{i=0}^{\infty
}\omega \left( x,y|\rho q^{i}\right) },  \label{ort1}
\end{eqnarray}%
is obtained from (\ref{gest_Q}) after an obvious adjustment to complex
parameters and application of Proposition \ref{reduce}, i).

Similarly we will consider the AW polynomials with two pairs of complex
conjugate parameters. We will denote new parameters $\rho _{1},$ $\rho _{2},$
$y,$ $z.$

We have the following proposition.

\begin{proposition}
With new parameters defined by (\ref{newparam}) we have i) 
\begin{equation*}
f_{n}(a,b,c,d,q)=\allowbreak \left( 1-q^{n}\right) \frac{\omega \left(
y,z|\rho _{1}\rho _{2}q^{n-1}\right) \left( \rho _{1}^{2}q^{n-1},\rho
_{2}^{2}q^{n-1},\rho _{1}^{2}\rho _{2}^{2}q^{n-2}\right) _{1}}{\left( \rho
_{1}^{2}\rho _{2}^{2}q^{2n-3}\right) _{3}\left( \rho _{1}^{2}\rho
_{2}^{2}q^{2n-2}\right) _{1}},
\end{equation*}

ii) 
\begin{gather*}
e_{n}(a,b,c,d,q)=\allowbreak \frac{2q^{n-1}}{\left( \rho _{1}^{2}\rho
_{2}^{2}q^{2n-2}\right) _{1}\left( \rho _{1}^{2}\rho _{2}^{2}q^{2n}\right)
_{1}}\times \\
(q\left( \rho _{1}y+\rho _{2}z\right) \left( 1-\rho _{1}^{2}\rho
_{2}^{2}q^{n-2}(1+q-q^{n+1}\right) )+ \\
\rho _{1}\rho _{2}\left( \rho _{1}z+\rho _{2}y\right) \left(
1-q^{n+1}-q^{n}+\rho _{1}^{2}\rho _{2}^{2}q^{2n-1}\right) ),
\end{gather*}

iii) $\left( ab,ac,ad,bc,bd,cd\right) _{1}\allowbreak =\allowbreak \left(
\rho _{1}^{2},\rho _{2}^{2}\right) _{1}\omega \left( y,z|\rho _{1}\rho
_{2}\right) ,$ $abcd\allowbreak =\allowbreak \rho _{1}^{2}\rho _{2}^{2}.$
\end{proposition}

Let us denote also AW polynomials $w_{n}\left( x|y,\rho _{1},z,\rho
_{2},q\right) $ with new complex parameters i.e. for for all $n\geq -1$ we
set: 
\begin{equation*}
w_{n}\left( x|y,\rho _{1},z,\rho _{2},q\right) =\alpha _{n}\left( x|\rho
_{1}e^{i\theta },\rho _{1}e^{-i\theta },\rho _{2}e^{i\eta },\rho
_{2}e^{-i\eta },q\right) .
\end{equation*}

\begin{corollary}
\label{gestAW} 
\begin{gather*}
f_{AW}(x|\rho _{1}e^{i\theta },\rho _{1}e^{-i\theta },\rho _{2}e^{i\eta
},\rho _{2}e^{-i\eta })= \\
f_{W}\left( x|y,\rho _{1},z,\rho _{2},q\right) =f_{h}\left( x|q\right) \frac{%
\left( \rho _{1}^{2},\rho _{2}^{2}\right) _{\infty }\prod_{j=0}^{\infty
}\omega \left( y,z|\rho _{1}\rho _{2}q^{j}\right) }{\left( \rho _{1}^{2}\rho
_{2}^{2}\right) _{\infty }\prod_{j=0}^{\infty }\omega \left( x,y|\rho
_{1}q^{j}\right) \omega \left( x,z|\rho _{2}q^{j}\right) }.
\end{gather*}
\end{corollary}

\begin{proof}
Follows (\ref{fAW}) with (\ref{newparam}) and the simplified by Proposition %
\ref{reduce} i) and ii).
\end{proof}

\begin{remark}
\label{rho=0}Notice that: i) for all $n\geq 0:w_{n}\left( x|y,0,\rho
_{2},z,q\right) \allowbreak =\allowbreak p_{n}\left( x|z,\rho _{2},q\right)
, $ where $p_{n}$ is defined is defined using polynomials $Q_{n}$ above and
satisfies 3-term recurrence (\ref{_3tr_p}),

ii) $f_{W}\left( x|y,0,z,\rho _{2},q\right) \allowbreak =\allowbreak
f_{p}\left( x|z,\rho _{2},q\right) .$
\end{remark}

\begin{proof}
$\rho _{1}\allowbreak =\allowbreak 0$ implies $a\allowbreak =\allowbreak
b\allowbreak =\allowbreak 0$ which leads to $f_{n}\left( 0,0,c,d,q\right)
\allowbreak =\allowbreak $\newline
$\left( 1-q^{n}\right) \left( 1-\rho _{2}^{2}q^{n-1}\right) ,$ and $%
e_{n}\left( 0,0,c,d,q\right) \allowbreak =\allowbreak 2q^{n}\rho _{2}z$. Now
we confront these values with (\ref{_3tr_p}) getting first assertion. The
second observation is trivial.
\end{proof}

We will use also the following useful Lemma.

\begin{lemma}
\label{conversion}For all $n,m\geq 0$:%
\begin{eqnarray}
&&\sum_{k=0}^{n}\sum_{j=0}^{m}\left[ \QATOP{n}{k}\right] _{q}\left[ \QATOP{m%
}{j}\right] _{q}\frac{\left( te^{i\left( -\theta +\eta \right) }\right)
_{k}\left( te^{i(\theta -\eta }\right) _{j}\left( te^{i\left( -\theta -\eta
\right) }\right) _{k+j}}{\left( t^{2}\right) _{k+j}}e^{i\left( 2k-n\right)
\theta }e^{i\left( 2j-m\right) \eta }  \label{wzor} \\
&=&\sum_{l=0}^{m}\left[ \QATOP{m}{l}\right] _{q}(-t)^{l}q^{\binom{l}{2}}%
\frac{h_{m-l}\left( z|q\right) p_{n+l}\left( y|z,t,q\right) }{\left(
t^{2}\right) _{n+l}},  \notag
\end{eqnarray}%
with $y\allowbreak =\allowbreak \cos \theta $ and $z\allowbreak =\allowbreak
\cos \eta .$
\end{lemma}

\begin{proof}
Is shifted to Section \ref{dowody}.
\end{proof}

As a corollary we have our main result.

\begin{theorem}
\label{main}For all $n\geq 0:$%
\begin{gather}
w_{n}\left( x|y,\rho _{1},z,\rho _{2},q\right) =  \label{ex_w_na_p} \\
\sum_{j=0}^{n}\left[ \QATOP{n}{j}\right] _{q}p_{j}\left( x|y,\rho
_{1},q\right) \frac{\rho _{2}^{n-j}\left( \rho _{1}^{2}q^{j}\right) _{n-j}}{%
\left( \rho _{1}^{2}\rho _{2}^{2}q^{n+j-1}\right) _{n-j}}g_{n-j}\left(
z|y,\rho _{1}\rho _{2}q^{n-1},q\right) ,  \notag
\end{gather}
\begin{equation}
p_{n}\left( x|y,\rho _{1},q\right) =\sum_{j=0}^{n}\left[ \QATOP{n}{j}\right]
_{q}w_{j}\left( x|y,\rho _{1},z,\rho _{2},q\right) \frac{\rho
_{2}^{n-j}\left( \rho _{1}^{2}q^{j}\right) _{n-j}}{\left( \rho _{1}^{2}\rho
_{2}^{2}q^{2j}\right) _{n-j}}p_{n-j}\left( z|y,\rho _{1}\rho
_{2}q^{j},q\right) ,  \label{ex_p_na_w}
\end{equation}%
where $g_{n}\left( z|y,\tau ,q\right) \allowbreak \ $is defined by (\ref{_g}%
).
\end{theorem}

\begin{proof}
Is shifted to Section \ref{dowody}.
\end{proof}

\begin{remark}
It should be stressed that although there exist formulae for 'connection
coefficients' between AW polynomials with different sets of parameters such
as (\cite{AW85}, (6.4)) or (\cite{IA}, (16.4.3)) the formulae (\ref%
{ex_p_na_w}) and (\ref{ex_w_na_p}) are by no means their direct or simple
consequences. Of course on the way we used (\ref{asw_na_aw}) and (\ref%
{aw_na_asc}) which may be regarded as more or less direct consequences of (%
\cite{AW85}, (6.4)). However as the proof of Lemma \ref{conversion} shows
the way from (\ref{asw_na_aw}) and (\ref{aw_na_asc}) to (\ref{ex_p_na_w})
and (\ref{ex_w_na_p}) is neither direct nor simple.
\end{remark}

As a corollary we have the following expansion of the Askey--Wilson density
with complex parameters.

\begin{corollary}
\label{AWexp}i)%
\begin{equation}
f_{W}\left( x|y,\rho _{1},z,\rho _{2},q\right) =f_{p}\left( x|y,\rho
_{1},q\right) \sum_{j\geq 0}\frac{\rho _{2}^{j}p_{j}\left( z|y,\rho _{1}\rho
_{2},q\right) }{\left( q,\rho _{1}^{2}\rho _{2}^{2}\right) _{j}}p_{j}\left(
x|y,\rho _{1},q\right) ,  \label{rozwiniecie1}
\end{equation}%
hence for for all $n\geq 1,\left\vert y\right\vert ,\left\vert z\right\vert
\leq 1,\left\vert \rho _{1}\right\vert ,\left\vert \rho _{2}\right\vert <1:$%
\begin{equation}
\int_{\lbrack -1,1]}p_{n}\left( x|y,\rho _{1},q\right) f_{W}\left( x|y,\rho
_{1},z,\rho _{2},q\right) dx=\frac{\rho _{2}^{n}\left( \rho _{1}^{2}\right)
_{n}}{\left( \rho _{1}^{2}\rho _{2}^{2}\right) _{n}}p_{n}(z|y,\rho _{1}\rho
_{2},q),  \label{calka1}
\end{equation}

ii) 
\begin{gather}
f_{p}\left( x|y,\rho _{1},q\right) =f_{W}\left( x|y,\rho _{1},z,\rho
_{2},q\right) \times  \label{rozwiniecie2} \\
\sum_{j\geq 0}\frac{\left( \rho _{1}^{2}\rho _{2}^{2}\right) _{2j}\rho
_{2}^{j}g_{j}\left( z|y,\rho _{1}\rho _{2}q^{j-1},q\right) }{\left( q,\rho
_{2}^{2})_{j}(\rho _{1}^{2}\rho _{2}^{2}q^{j-1}\right)
_{j}^{2}\prod_{k=1}^{j}\omega \left( y,z|\rho _{1}\rho _{2}q^{k-1}\right) }%
w_{j}\left( x|y,\rho _{1},z,\rho _{2},q\right) ,  \notag
\end{gather}

hence for all $n\geq 1,\left\vert y\right\vert ,\left\vert z\right\vert \leq
1,\left\vert \rho _{1}\right\vert ,\left\vert \rho _{2}\right\vert <1:$%
\begin{equation}
\int_{\lbrack -1,1]}w_{n}(x|y,\rho _{1},z,\rho _{2},q)f_{p}\left( x|y,\rho
_{1},q\right) dx=\frac{\rho _{2}^{n}\left( \rho _{1}^{2}\right) _{n}}{\left(
\rho _{1}^{2}\rho _{2}^{2}q^{n-1}\right) _{n}}g_{n}\left( z|y,\rho _{1}\rho
_{2}q^{n-1},q\right) .  \label{calka2}
\end{equation}
\end{corollary}

\begin{proof}
We use idea of density expansion presented at the beginning of the paper. To
get i) we use (\ref{series}), (\ref{ex_p_na_w}) and (\ref{alfa^2}) with $%
d=c=0$. ii) We use (\ref{series}), (\ref{ex_w_na_p}) and (\ref{alfa^2}) .
\end{proof}

\begin{remark}
Expansion given in assertion i) was (with slightly different scaling)
obtained by the similar method in \cite{Szabl-ker} or as a particular case
of an other expansion in \cite{Szabl-KS}.
\end{remark}

\begin{remark}
\label{generalization}Let us consider the following correspondence between
two groups of polynomials. To the first group let us take polynomials $%
\{h_{n}(x|q),$ $b_{n}\left( y|q\right) ,$\newline
$p_{n}\left( x|y,\rho ,q\right) \}$ and to the second $\{p_{n}\left(
x|y,\rho _{1},q\right) ,$ $g_{n}\left( z|y,\rho _{1}\rho
_{2}q^{n-1},q\right) ,$ \newline
$w_{n}\left( x|y,\rho _{1},z,\rho _{2},q\right) \}$ then we see that formula
(\ref{ex_w_na_p}) is a generalization of Lemma \ref{zal_b_h}, iii) while (%
\ref{ex_p_na_w}) is a generalization of Lemma \ref{zal_b_h}, iv). Besides (%
\ref{rozwiniecie1}) is a generalization of the Poisson--Mehler expansion
formula 
\begin{equation}
f_{p}\left( x|y,\rho ,q\right) =f_{h}\left( x|q\right) \sum_{j=0}^{\infty }%
\frac{\rho ^{j}}{\left( q\right) _{j}}h_{j}\left( x|q\right) h_{j}\left(
y|q\right) ,  \label{PM}
\end{equation}%
which has many proofs see e.g. \cite{bressoud}, \cite{IA} (page 324, Thm.
13.1.6) or \cite{Szabl-exp} where the presented above idea of the density
expansion is applied. Similarly (\ref{rozwiniecie2}) is a generalization of 
\cite{Szabl-exp}, (5.3). Notice that following Remark \ref{rho=0} we see
that (\ref{rozwiniecie1}) as well (\ref{ex_w_na_p}) reduce to (\ref{PM}) and
Lemma \ref{zal_b_h}, iii) when it is set $\rho _{1}\allowbreak =\allowbreak
0.$ Similarly (\ref{rozwiniecie2}) and (\ref{ex_p_na_w}) reduce to (\cite%
{Szabl-exp}, (5.3)) and Lemma \ref{zal_b_h}, iv) when we set $\rho
_{1}\allowbreak =\allowbreak 0.$
\end{remark}

\begin{corollary}
\label{identies}i) For all $n\geq 1,0\leq k<n,z,y,t\in \mathbb{R}:$%
\begin{equation*}
\sum_{j=0}^{n-k}\QATOPD[ ] {n-k}{j}_{q}\frac{p_{j}\left( z|y,tq^{k},q\right) 
}{\left( t^{2}q^{2k}\right) _{j}}\frac{g_{n-k-j}\left( z|y,tq^{n-1},q\right) 
}{\left( t^{2}q^{n+j+k-1}\right) _{n-k-j}}\allowbreak =\allowbreak 0,
\end{equation*}

ii) for all $n\geq 1,0\leq k<n,z,y,t\in \mathbb{R}:$%
\begin{equation*}
\sum_{m=0}^{n-k}\QATOPD[ ] {n-k}{m}_{q}\frac{p_{n-k-m}\left(
z|y,tq^{m+k},q\right) g_{m}(z|y,tq^{m+k-1},q)}{\left( t^{2}q^{2m+2k}\right)
_{n-k-m}\left( t^{2}q^{m+2k-1}\right) _{m}}\allowbreak =\allowbreak 0.
\end{equation*}
\end{corollary}

\begin{proof}
Is shifted to Section \ref{dowody}.
\end{proof}

\begin{remark}
Notice that apart from the two above mentioned identities involving
polynomials $\left\{ p_{n}\right\} $ and $\left\{ g_{n}\right\} $ we have
also the one exposed in Proposition \ref{odwr_p}, iv). All these identities
are different.
\end{remark}

\section{Towards probabilistic interpretation\label{prob}}

To get nice probabilistic interpretation of Askey--Wilson density and
polynomials we have to make certain re-scaling and re-normalization of all
polynomials that appeared above. Let us introduce the following sets:%
\begin{equation*}
S\left( q\right) =\left\{ 
\begin{array}{ccc}
\left[ -2/\sqrt{1-q},2/\sqrt{1-q}\right] & if & \left\vert q\right\vert <1,
\\ 
\mathbb{R} & if & q=1.%
\end{array}%
\right.
\end{equation*}

We will consider the following polynomials: \newline
\begin{gather*}
H_{n}\left( x|q\right) \allowbreak =\allowbreak h_{n}\left( x\sqrt{1-q}%
/2|q\right) /\left( 1-q\right) ^{n/2}, \\
P_{n}\left( x|y,\rho ,q\right) =p_{n}\left( x\sqrt{1-q}/2|y\sqrt{1-q}/2,\rho
,q\right) /\left( 1-q\right) ^{n/2}, \\
A_{n}\left( x|y,\rho _{1},z,\rho _{2},q\right) \allowbreak =\allowbreak
w_{n}\left( x\sqrt{1-q}/2|y\sqrt{1-q}/2,\rho _{1},z\sqrt{1-q}/2,\rho
_{2},q\right) /\left( 1-q\right) ^{n/2}, \\
B_{n}\left( x|q\right) \allowbreak =\allowbreak b_{n}\left( x\sqrt{1-q}%
/2|q\right) /\left( 1-q\right) ^{n/2}, \\
~G_{n}\left( x|y,\rho ,q\right) \allowbreak =\allowbreak g_{n}\left( x\sqrt{%
1-q}/2|y\sqrt{1-q}/2,\rho ,q\right) .
\end{gather*}

Notice that now polynomials $H_{n},$ $P_{n},$ and $A_{n}$ are monic.

One can easily check that these new polynomials satisfy the following
three-term recurrence relations: 
\begin{gather}
H_{n+1}\left( x|q\right) \allowbreak =\allowbreak xH_{n}\left( x|q\right)
-[n]_{q}H_{n-1}(x|q),  \label{pH} \\
P_{n+1}\left( x|y,\rho ,q\right) =(x-\rho yq^{n})P_{n}\left( x|y,\rho
,q\right) -[n]_{q}(1-\rho ^{2}q^{n-1})P_{n-1}\left( x|y,\rho ,q\right) ,
\label{pASC} \\
A_{n+1}\left( x|y,\rho _{1},z,\rho _{2},q\right) =(x-\beta _{n})A_{n}\left(
x|y,\rho _{1},z,\rho _{2},q\right) -\gamma _{n}A_{n-1}\left( x|y,\rho
_{1},z,\rho _{2},q\right) ,  \label{pA}
\end{gather}%
where we use Pochhammer symbol and (\ref{gr_w}) for the brevity and where:%
\begin{gather}
\beta _{n}\allowbreak =\allowbreak \frac{q^{n}\left( y\rho _{1}+z\rho
_{2}\right) }{(1-\rho _{1}^{2}\rho _{2}^{2}q^{2n-2})(1-\rho _{1}^{2}\rho
_{2}^{2}q^{2n})}\times  \label{beta_n} \\
((1-q^{n-2}(1+q-q^{n+1})\rho _{1}^{2}\rho _{2}^{2})+q^{n-1}\rho _{1}\rho
_{2}\left( z\rho _{1}+y\rho _{2}\right) (1-q^{n+1}-q^{n}+\rho _{1}^{2}\rho
_{2}^{2}q^{2n-1})),  \notag
\end{gather}%
\begin{equation}
\gamma _{n}\allowbreak =\allowbreak \frac{\left[ n\right] _{q}\left( \rho
_{1}^{2}q^{n-1}\right) _{1}\left( \rho _{2}^{2}q^{n-1}\right) _{1}\left(
\rho _{1}^{2}\rho _{2}^{2}q^{n-2}\right) _{1}\omega \left( y\frac{\sqrt{1-q}%
}{2},z\frac{\sqrt{1-q}}{2}|\rho _{1}\rho _{2}q^{n-1}\right) }{\left( \rho
_{1}^{2}\rho _{2}^{2}q^{2n-3}\right) _{3}\left( \rho _{1}^{2}\rho
_{2}^{2}q^{2n-2}\right) _{1}},  \label{c_n}
\end{equation}%
with $H_{-1}\left( x|q\right) \allowbreak =\allowbreak H_{-1}\left(
x|a,q\right) \allowbreak =\allowbreak P_{-1}\left( x|y,\rho ,q\right)
\allowbreak =\allowbreak B_{-1}\left( x|q\right) \allowbreak =\allowbreak
G_{-1}\left( x|y,\rho ,q\right) \allowbreak =\allowbreak A_{-1}\left(
x|y,\rho _{1},z,\rho _{2},q\right) \allowbreak =\allowbreak 0,$ $H_{0}\left(
x|q\right) \allowbreak =\allowbreak H_{0}\left( x|a,q\right) \allowbreak
=\allowbreak P_{0}\left( x|y,\rho ,q\right) \allowbreak =\allowbreak
B_{0}\left( x|q\right) \allowbreak =\allowbreak G_{0}\left( x|y,\rho
,q\right) \allowbreak =\allowbreak A_{0}\left( x|y,\rho _{1},z,\rho
_{2},q\right) \allowbreak =\allowbreak 1.$

Besides this change will allow passing with $q$ to $1^{-}.$ Hence we will
have two special cases associated with two extreme values of $q$. The first
one $q\allowbreak =\allowbreak 0$ (so called free case) and the second $%
q\allowbreak =\allowbreak 1$ (classical case) that will allow to support
intuition. To describe those cases swiftly let us denote by $H_{n}\left(
x\right) $ monic polynomials that are orthogonal with respect to the measure
with the density $\exp \left( -x^{2}/2\right) /\sqrt{2\pi }.$ and by $%
U_{n}\left( x\right) $ let us denote Chebyshev polynomials of the second
kind . Thus these polynomials are defined by the following three-term
recurrence relations (see e.g. \cite{KLS}, (9.15.3) with obvious adaptation
to the density $\exp \left( -x^{2}/2\right) /\sqrt{2\pi }$ and (9.8.4)):%
\begin{equation}
H_{n+1}\left( x\right) =xH_{n}\left( x\right) -nH_{n-1}\left( x\right)
,U_{n+1}\left( x\right) =2xU_{n}\left( x\right) -U_{n-1}\left( x\right) ,
\label{zwH}
\end{equation}%
with $H_{-1}\left( x\right) \allowbreak =\allowbreak U_{-1}\left( x\right)
\allowbreak =\allowbreak 0,$ $H_{0}\left( x\right) \allowbreak =\allowbreak
U_{0}\left( x\right) \allowbreak =\allowbreak 1.$

\begin{proposition}
\label{wilomiany 01}For all $n\geq 0:$ i) $H_{n}\left( x|1\right)
\allowbreak =\allowbreak H_{n}\left( x\right) ,~H_{n}\left( x|0\right)
\allowbreak =\allowbreak U_{n}\left( x/2\right) ,$

ii) 
\begin{eqnarray*}
P_{n}\left( x|y,\rho ,1\right) \allowbreak &=&\allowbreak H_{n}\left( \frac{%
x-\rho y}{\sqrt{1-\rho ^{2}}}\right) \left( 1-\rho ^{2}\right) ^{n/2}, \\
P_{n}\left( x|y,\rho ,0\right) \allowbreak &=&\allowbreak U_{n}\left(
x/2\right) \allowbreak -\allowbreak \rho yU_{n-1}\left( x/2\right) +\rho
^{2}U_{n-2}\left( x/2\right) ,
\end{eqnarray*}

iii) 
\begin{gather*}
A_{n}\left( x|y,\rho _{1},z,\rho _{2},1\right) \allowbreak =\allowbreak
\left( \frac{\left( 1-\rho _{1}^{2}\right) \left( 1-\rho _{2}^{2}\right) }{%
1-\rho _{1}^{2}\rho _{2}^{2}}\right) ^{n/2}\times \\
H_{n}\left( \frac{x(1-\rho _{1}^{2}\rho _{2}^{2})-\rho _{1}\left( 1-\rho
_{2}^{2}\right) y-\rho _{2}(1-\rho _{1}^{2})z}{\sqrt{\left( 1-\rho
_{1}^{2}\right) \left( 1-\rho _{2}^{2}\right) \left( 1-\rho _{1}^{2}\rho
_{2}^{2}\right) }}\right) , \\
A_{n}\left( x|y,\rho _{1},z,\rho _{2},0\right) \allowbreak =\allowbreak
U_{n}\left( x/2\right) \allowbreak \allowbreak -\allowbreak \left( \rho
_{1}y+\rho _{2}z\right) U_{n-1}\left( x/2\right) \allowbreak \\
+\allowbreak (\rho _{1}^{2}+\rho _{2}^{2}+yz\rho _{1}\rho _{2})U_{n-2}\left(
x/2\right) \allowbreak -\allowbreak \rho _{1}\rho _{2}\left( \rho _{1}z+\rho
_{2}y\right) U_{n-3}\left( x/2\right) \allowbreak +\allowbreak \rho
_{1}^{2}\rho _{2}^{2}U_{n-4}\left( x/2\right) .
\end{gather*}
\end{proposition}

\begin{proof}
First of all let us notice that $\left[ n\right] _{1}\allowbreak
=\allowbreak n$ and $\left[ n\right] _{0}\allowbreak =\allowbreak 1.$ So (%
\ref{pH}) for $q\allowbreak =\allowbreak 1$ and $q\allowbreak =\allowbreak 0$
reduces to (\ref{zwH}). Further for $q\allowbreak =\allowbreak 1,$ $\beta
_{n}\allowbreak =\allowbreak \frac{y\rho _{1}\left( 1-\rho _{2}^{2}\right)
+z\rho _{2}\left( 1-\rho _{1}^{2}\right) }{1-\rho _{1}^{2}\rho _{2}^{2}}$
and $\gamma _{n}\allowbreak =\frac{n\left( 1-\rho _{1}^{2}\right) \left(
1-\rho _{2}^{2}\right) }{1-\rho _{1}^{2}\rho _{2}^{2}}\allowbreak $ and for $%
q\allowbreak =\allowbreak 0,$ $\beta _{n}\allowbreak =\allowbreak 0$ and $%
\gamma _{n}\allowbreak =1$. ii) Follows (\ref{pASC}) and for $q\allowbreak
=\allowbreak 1$ (\ref{zwH}) modified in a standard way to fit (\ref{zwH}).
The case $q\allowbreak =\allowbreak 0$ follows directly Lemma \ref{B and H},
i). iii) Follows the fact that $\beta _{n}$ for $q\allowbreak =\allowbreak 1$
is equal to $\frac{y\rho _{1}\left( 1-\rho _{2}^{2}\right) +z\rho _{2}\left(
1-\rho _{1}^{2}\right) }{1-\rho _{1}^{2}\rho _{2}^{2}}$ does not depend on $%
n $ while for $q\allowbreak =\allowbreak 0$ is equal to $0$ for $n\geq 2.$
For $q\allowbreak =\allowbreak 1$ $\gamma _{n}$ is equal to $\frac{n\left(
1-\rho _{1}^{2}\right) \left( 1-\rho _{2}^{2}\right) }{1-\rho _{1}^{2}\rho
_{2}^{2}}$ -proportional to $n$ while for $q\allowbreak =\allowbreak 0$ to $%
1 $ for $n\geq 2.$ For $q\allowbreak =\allowbreak 1$ it is easy to notice
that three-term recurrence relation satisfied by polynomials $A_{n}$ is a
modified one given by (\ref{zwH}). For $q\allowbreak =\allowbreak 0$ one can
reason in two ways. One way of reasoning is to observe that the density
given by (\ref{fpAW}) has a form of the ratio of $\sqrt{4-x^{2}}$ divided by
certain polynomial of order $4.$ Hence following Proposition 1 of \cite%
{Szabl-exp} we deduce that polynomials that are orthogonal with respect to
the density of this form must be of the form of linear combination of
successive $4+1\allowbreak =\allowbreak 5$ polynomials that are orthogonal
with respect to the density $\sqrt{4-x^{2}}/2\pi $ i.e. Chebyshev
polynomials of the second kind. Or we can notice that the three-term
recurrence relation satisfied by AW polynomials for $q\allowbreak
=\allowbreak 0$ is the same as that of the Chebyshev polynomials of the
second kind with different initial conditions. Now it remains to find
coefficients of the linear combination of the Chebyshev polynomials of the
second kind.
\end{proof}

Adapted to current setting Lemma \ref{zal_b_h} has the following form

\begin{lemma}
\label{B and H}For all $n\geq 0,$ $x,y,\rho \in \mathbb{R},$

i) $P_{n}\left( x|y,\rho ,q\right) =\sum_{j=0}^{n}\QATOPD[ ] {n}{j}_{q}\rho
^{n-j}H_{j}\left( x|q\right) B_{n-j}\left( y|q\right) ,$

ii) $H_{n}\left( x|q\right) \allowbreak =\allowbreak \sum_{j=0}^{n}\QATOPD[ ]
{n}{j}_{q}\rho ^{n-j}H_{n-j}\left( y|q\right) P_{j}\left( x|y,\rho ,q\right)
.$
\end{lemma}

\begin{proof}
i) follows Lemma \ref{zal_b_h}, iii) (shown in \cite{bms}). ii) follows
almost Lemma \ref{zal_b_h}, iv).
\end{proof}

Polynomials $\left\{ H_{n}\left( x|q\right) \right\} $ and $\left\{
P_{n}(x|y,\rho ,q\right\} $ are orthogonal with respect to measures having
respectively the following densities

\begin{eqnarray}
f_{N}\left( x|q\right) \allowbreak &=&\allowbreak f_{h}\left( \frac{x\sqrt{%
1-q}}{2}|q\right) \frac{\sqrt{1-q}}{2},  \label{fN} \\
~f_{CN}\left( x|y,\rho ,q\right) \allowbreak &=&\allowbreak f_{N}\left(
x|q\right) \frac{\left( \rho ^{2}\right) _{\infty }}{\prod_{j=0}^{\infty
}\omega \left( x\frac{\sqrt{1-q}}{2},y\frac{\sqrt{1-q}}{2}|\rho q^{j}\right) 
},  \label{fCN}
\end{eqnarray}%
$x,y\in S\left( q\right) .$ We will call distribution with the density $%
f_{N} $ $q-$Normal while with $f_{CN}$ conditional $q-$Normal.

By Corollary \ref{gestAW} polynomials $\left\{ A_{n}\right\} $ are
orthogonal with respect to measure with the density : 
\begin{equation}
f_{C2N}\left( x|y,\rho _{1},z,\rho _{2},q\right) \allowbreak =\allowbreak 
\frac{f_{CN}\left( y|x,\rho _{1},q\right) f_{CN}\left( x|z,\rho
_{2},q\right) }{f_{CN}\left( y|z,\rho _{1}\rho _{2},q\right) }.  \label{fpAW}
\end{equation}%
The subindex $C2N$ refers to the conditional Normal with two fixed ends (for
clearer interpretation see Theorem \ref{3zmienne}, below).

We have several easy observations that are believed to help intuition which
we will put two remarks one concerning densities the second polynomials:

\begin{proposition}
\label{densities 01} $f_{N}\left( x|1\right) \allowbreak $ and $f_{CN}\left(
x|y,\rho ,1\right) \mathbb{\ }$are Normal densities with appropriate
parameters $f_{N}\left( x|0\right) \allowbreak =\allowbreak \frac{1}{2\pi }%
\sqrt{4-x^{2}},~\left\vert x\right\vert \leq 2,$ (Wigner distribution with
parameter $2$), $f_{CN}\left( x|y,\rho ,0\right) \allowbreak =\allowbreak 
\frac{\left( 1-\rho ^{2}\right) \sqrt{4-x^{2}}}{2\pi (\rho
^{2}(x^{2}+y^{2})-\rho xy(1+\rho ^{2})+(1-\rho ^{2})^{2})}$ (Kesten--McKey
distribution). From (\ref{fpAW}) it follows that $f_{C2N}\left( x|y,\rho
_{1},z,\rho _{2},1\right) $ is the density of certain Normal distribution
with appropriate parameters while

\begin{equation*}
f_{C2N}\left( x|y,\rho _{1},z,\rho _{2},0\right) \allowbreak =\allowbreak 
\frac{(1-\rho _{1}^{2})(1-\rho _{2}^{2})\sqrt{4-x^{2}}\omega \left( \frac{1}{%
2}z,\frac{1}{2}y|\rho _{1}\rho _{2}\right) }{2\pi (1-\rho _{1}^{2}\rho
_{2}^{2})\omega \left( \frac{1}{2}x,\frac{1}{2}y|\rho _{1}\right) \omega
\left( \frac{1}{2}x,\frac{1}{2}z|\rho _{2}\right) },
\end{equation*}%
\newline
for $\left\vert x\right\vert ,\left\vert y\right\vert ,\left\vert
z\right\vert \leq 2,$ $\left\vert \rho _{1}\right\vert ,\rho _{2}|<1.$
\end{proposition}

\begin{proof}
$f_{N}.$ For $q\allowbreak =\allowbreak 1$ see \cite{IA}, (13.1.33), for $%
q\allowbreak =\allowbreak 0$ one can see directly. $f_{CN}$ for $%
q\allowbreak =\allowbreak 1$ see \cite{IA}, (13.1.33) and necessary
adaptation of parameters stemming from Lemma \ref{B and H}, i) and the
properties of polynomials $B_{n}$ given e.g. in \cite{bms}, for $q=0$ we get
it directly from (\ref{fN}) and (\ref{fCN}). $f_{C2N}$ for $q\allowbreak
=\allowbreak 1$ it follows (\ref{fpAW}) and formula for $f_{CN}$, for $%
q\allowbreak =\allowbreak 0$ we apply (\ref{fN}) and (\ref{fCN}) and direct
calculation. Now we are ready to present probabilistic applications of the
densities $f_{N}$ , $f_{CN}$ and $f_{C2N}$ as well as polynomials $H_{n},$ $%
P_{n}$ , $A_{n}.$
\end{proof}

Let us consider $3$ random variable say $X,$ $Y,$ $Z$ their joint density is
equal to $f_{N}\left( y|q\right) f_{CN}\left( x|y,\rho _{1},q\right)
f_{CN}\left( z|x,\rho _{2},q\right) .$ Another words we assume that $Y,X,Z$
form a finite (of length $3)$ Markov chain.

\begin{theorem}
\label{3zmienne}For almost all $x,y,z\in S\left( q\right) $ and all $n\geq 1$
we have:

i) $X\allowbreak \sim \allowbreak Y\allowbreak \sim \allowbreak Z\allowbreak
\sim \allowbreak f_{N},$ $\left( Y,X\right) \allowbreak \sim \allowbreak
f_{N}\left( x|q\right) f_{CN}\left( y|x,\rho _{1},q\right) ,$ \newline
$(Z,X)\allowbreak \sim \allowbreak f_{N}\left( x|q\right) f_{CN}\left(
z|x,\rho _{2},q\right) ,$ $\left( Z,Y\right) \allowbreak \sim \allowbreak
f_{N}\left( z|q\right) f_{CN}\left( y|z,\rho _{1}\rho _{2},q\right) ,$

ii) $Y|\left( X=x,Z=z\right) \allowbreak \sim \allowbreak f_{CN}(y|x,\rho
_{1},q),$ $Z|\left( X=x,Y=y\right) \allowbreak \sim f_{CN}(z|x,\rho _{2},q),$
$\allowbreak X|(Y=y,Z=z)\allowbreak \sim \allowbreak f_{C2N}\left( x|y,\rho
_{1},z,\rho _{2},q\right) ,$

iii) for $n\geq 1$, $\mathbb{E}\left( A_{n}\left( X|y,\rho _{1},z,\rho
_{2},q\right) |Y=y,Z=z\right) \allowbreak =\allowbreak 0,$ \newline
$\mathbb{E(}P_{n}(X|y,\rho _{1},q)|Y=y,Z=z)\allowbreak =\frac{\rho
_{2}^{n}\left( \rho _{1}^{2}\right) _{n}}{\left( \rho _{1}^{2}\rho
_{2}^{2}\right) _{n}}P_{n}\left( z|y,\rho _{1}\rho _{2},q\right) ,$

iv) for $n\geq 1$, $\mathbb{E}\left( P_{n}\left( X|y,\rho _{1},q\right)
|Y=y\right) \allowbreak =\allowbreak 0,$ \newline
$\mathbb{E(}A_{n}\left( X|y,\rho _{1},z,\rho _{2},q\right) |Y=y)=\frac{\rho
_{2}^{n}\left( \rho _{1}^{2}\right) _{n}}{\left( \rho _{1}^{2}\rho
_{2}^{2}q^{n-1}\right) _{n}}G_{n}\left( z|y,\rho _{1}\rho
_{2}q^{n-1},q\right) .$
\end{theorem}

\begin{proof}
The two dimensional densities given in assertion i) and conditional
densities given in assertion ii) follow the following observations: \newline
First one $\int_{S\left( q\right) }f_{N}\left( x|q\right) f_{CN}\left(
y|x,\rho _{1},q\right) dx\allowbreak =\allowbreak f_{N}\left( y|q\right) $
which follows (\ref{fN} and \ref{fCN}) from which it follows that $%
f_{N}\left( y|q\right) f_{CN}\left( x|y,t,q\right) \allowbreak =\allowbreak
f_{N}\left( x|q\right) f_{CN}\left( y|x,t,q\right) .$ \newline
Second one $\int_{S\left( q\right) }f_{CN}\left( x|y,\rho _{1},q\right)
f_{CN}\left( y|z,\rho _{2},q\right) dy\allowbreak =\allowbreak f_{CN}\left(
x|z,\rho _{1}\rho _{2},q\right) $ for $\left\vert \rho _{1}\right\vert
,\left\vert \rho _{2}\right\vert <1$ which was proved in \cite{bms}. Third
one states that the conditional density $X|(Y=y,Z=z)$ is equal to the ratio
of joint density divided by the marginal density of $(Y,Z)$ given in
assertion i). iii) and iv) follow the fact that polynomials $A_{n}$ and $%
P_{n}$ are respectively orthogonal polynomials of $f_{C2N}$ and $f_{CN}$.
The remaining statements are in fact equivalent to assertions of Corollary %
\ref{AWexp}.
\end{proof}

Let us remark that assertions of Lemma \ref{3zmienne} do list only the most
important probabilistic aspects of polynomials and densities mentioned in
this paper. The others can be deduced from many other properties that were
presented in Section \ref{complex}. Let us mention for historical reasons
that first such probabilistic aspects appeared in the excellent paper of Bo%
\.{z}ejko et all \cite{Bo} where two stochastic processes that constitute
generalization of ordinary Wiener and Ornstein--Uhlenbeck processes
appeared. This was done in non-commutative probability context however is
important also in for classical probability. In \cite{Bo} there appeared
re-scaled densities $f_{h}$ and $f_{p}$ that this densities $f_{N}$ and $%
f_{CN}.$ Further analysis of the stochastic properties of these processes
was done recently in \cite{Szab-OU-W}.

Generalization of the presented above finite three-step Markov chain was
done in \cite{Szabl-mult}.

An attempt to overcome the Markov scheme was done in \cite{Szabl-KS}.
Unfortunately it was an unsuccessful attempt in the sense that it turned out
that it is impossible to construct true generalization of the
Kibble--Slepian formula that would lead to $3-$dimensional joint density
that would have marginals of the form \newline
$f_{N}\left( x|q\right) f_{CN}\left( y|x,\rho ,q\right) $ and many other
properties described in Lemma \ref{3zmienne} (except of course vii) and
viii)).

Let us recall that in 2001 W. Bryc in \cite{Bryc2001S} have constructed a
stationary Markov chain $\left\{ X_{n}\right\} _{n=-\infty }^{\infty }$
satisfying simple intuitive conditions imposed on its first two conditional
moments. It turned out that one- dimensional distribution of this chain has
density $f_{N}$ while its transitional distribution has density $f_{CN}.$
Let us mention that parameter $q$ does not appear in the original
probabilistic description of the chain. It has to be defined as a certain
rational function of parameters describing these first two conditional
moments. For details we refer interested readers to the original paper of
Bryc \cite{Bryc2001S}. Some recollection of known and presentation of new
properties of these Markov processes constructed by W. Bryc is also done in 
\cite{bms},\cite{Szab-OU-W},\cite{Szabl-mult}.

Finally let us mention also the fact that polynomials and densities analyzed
in this paper appear also in the so called 'quadratic harnesses context'.
For detail see \cite{BryMaWe},\cite{BryWe},\cite{bryc05}.

\section{Proofs\label{dowody}}

\begin{proof}[Proof of Proposition \protect\ref{wsp}]
Following \cite{KLS}, (14.14) let us define 
\begin{eqnarray*}
A_{n} &=&\frac{(1-abq^{n})(1-acq^{n})(1-adq^{n})(1-abcdq^{n-1})}{%
a(1-abcdq^{2n-1})(1-ancdq^{2n})}, \\
C_{n} &=&\frac{a(1-q^{n})(1-bcq^{n-1})(1-bdq^{n-1})(1-cdq^{n-1})}{%
(1-abcdq^{2n-1})(1-ancdq^{2n+1})}.
\end{eqnarray*}%
Then by straightforward calculations we show that 
\begin{equation*}
a+a^{-1}-A_{n}-C_{n}\allowbreak =\allowbreak e_{n}(a,b,c,d,q)
\end{equation*}%
and 
\begin{equation*}
A_{n-1}C_{n}=f_{n}(a,b,c,d,q).
\end{equation*}%
\newline
\end{proof}

\begin{proof}[Proof of Lemma \protect\ref{conAW-c2h}]
Let us consider $a\allowbreak =\allowbreak 0.$ Then we get $%
f_{n}(0,b,c,d,q)\allowbreak =\allowbreak
(1-q^{n})(1-bcq^{n-1})(1-bdq^{n-1})(1-cdq^{n-1})$ and $e_{n}(0,b,c,d,q)%
\allowbreak =\allowbreak q^{n}(b+c+d)\allowbreak +\allowbreak
bcdq^{n-1}(1-q^{n+1}-q^{n}).$ To get (\ref{a_na_c}) we apply formula (\cite%
{AW85}, (6.4)) that has two parameters $\alpha $ and $a$ values with $\
\alpha \allowbreak =\allowbreak a,$ $a\allowbreak =\allowbreak 0$. One has
to be aware that this formula is valid for polynomials $p_{n}$. We take into
account (\ref{correction}) and then we use formula $b^{n}(c/b)_{n}%
\allowbreak =\allowbreak b^{n}\prod_{k=0}^{n-1}(1-cq^{k}/b)\allowbreak
=\allowbreak \prod_{k=0}^{n}(b-cq^{k}).$ For $b=0$ we set $%
b^{n}(c/b)_{n}\allowbreak \allowbreak =\allowbreak (-c)^{n}q^{\binom{n}{2}}.$
Having this we have in our setting $a^{n-k}\left( \alpha /a\right)
_{n-k}\allowbreak =\allowbreak \left( -\alpha \right) ^{n-k}q^{\binom{n-k}{2}%
}$ when $a\allowbreak =\allowbreak 0.$ Hence we are getting:%
\begin{eqnarray*}
c_{k,n}\allowbreak &=&\allowbreak \left( -1\right) ^{k}q^{nk-\binom{k}{2}%
}\left( q^{-n}\right) _{k}q^{\binom{n-k}{2}}\times \\
&&\frac{(-a)^{n-k}\left( abcdq^{n-1}\right) _{k}\left(
bcq^{k},bdq^{k},cdq^{k}\right) _{n-k}}{\left( q\right) _{k}}\frac{1}{\left(
abcdq^{n-1}\right) _{n}}.
\end{eqnarray*}%
Now notice that $\left( q^{-n}\right) _{k}\allowbreak =\allowbreak \left(
-1\right) ^{k}q^{-nk+\binom{k}{2}}\QATOPD[ ] {n}{k}_{q}\left( q\right) _{k}.$
Hence we have:%
\begin{equation*}
c_{k,n}\allowbreak =\allowbreak \left[ \QATOP{n}{k}\right] _{q}q^{\binom{n-k%
}{2}}\left( -a\right) ^{n-k}\frac{\left( bcq^{k},bdq^{k},cdq^{k}\right)
_{n-k}}{\left( abcdq^{n+k-1}\right) _{n-k}}.
\end{equation*}%
To get (\ref{c_na_a}) we use the same formula but with $\alpha \allowbreak
=\allowbreak 0$.
\end{proof}

\begin{proof}[Proof of the Proposition \protect\ref{odwr_p}]
i) Let us notice that using (\ref{b_h}) and Lemma \ref{zal_b_h}, ii) we get 
\begin{gather*}
p_{n}\left( x|y,\rho ,q^{-1}\right) \allowbreak =\allowbreak \sum_{j=0}^{n}%
\QATOPD[ ] {n}{j}_{q^{-1}}\rho ^{n-j}h_{j}\left( x|q^{-1}\right)
b_{n-j}\left( y|q^{-1}\right) \allowbreak \\
=\allowbreak \sum_{j=0}^{n}\QATOPD[ ] {n}{j}_{q}q^{j(j-n)}\rho
^{n-j}(-1)^{j}\left( -1\right) ^{n-j}q^{-\binom{j}{2}-\binom{n-j}{2}%
}b_{j}\left( x|q\right) h_{n-j}\left( y|q\right) \allowbreak \\
=\allowbreak q^{-\binom{n}{2}}(-1)^{n}\sum_{j=0}^{n}\QATOPD[ ] {n}{j}%
_{q}\rho ^{n-j}b_{j}\left( x|q\right) h_{n-j}\left( y|q\right) \\
=q^{-\binom{n}{2}}(-1)^{n}\rho ^{n}p_{n}\left( y|x,\rho ^{-1},q\right) .
\end{gather*}
Since $\binom{j}{2}\allowbreak +\allowbreak \binom{n-j}{2}-j(j-n)\allowbreak
=\allowbreak \binom{n}{2}.$

ii) We have by (\ref{_3tr_p}) 
\begin{gather*}
g_{n+1}\left( x|y,\rho ,q\right) \allowbreak =\allowbreak \rho
^{n+1}(2(y-\rho ^{-1}q^{n}x)p_{n}\left( y|x,\rho ^{-1},q\right) - \\
(1-q^{n})(1-\rho ^{-2}q^{n-1})p_{n-1}\left( y|x,\rho ^{-1}q\right) )= \\
-2(q^{n}x-\rho y)g_{n}\left( x|y,\rho ,q\right) -(1-q^{n})(\rho
^{2}-q^{n-1})g_{n-1}(x|y,\rho ,q).
\end{gather*}

iii) Assume $\rho \neq 0.$ We have $\sum_{j=0}^{\infty }\frac{t^{n}}{\left(
q\right) _{n}}g_{n}\left( x|y,\rho ,q\right) \allowbreak =\allowbreak
\sum_{j=0}^{\infty }\frac{\rho ^{n}t^{n}}{\left( q\right) _{n}}p_{n}\left(
y|x,\rho ^{-1},q\right) \allowbreak =\allowbreak \varphi _{p}\left( y,\rho
t|x,\rho ^{-1},q\right) \allowbreak =\allowbreak \frac{\varphi _{h}\left(
y,\rho t|q\right) }{\varphi _{h}\left( x,t|q\right) }\allowbreak
=\allowbreak 1/\varphi _{p}\left( x,t|y,\rho ,q\right) $ by (\ref{chp2}).
iv) Follows directly iii).
\end{proof}

\begin{proof}[Proof of Corollary \protect\ref{conection}]
We start with (\ref{a_na_c}), then we use (\ref{c_na_a}) with $\allowbreak
d\allowbreak =\allowbreak 0$: 
\begin{gather*}
\alpha _{n}(x|a,b,c,d,q)\allowbreak =\allowbreak \sum_{j=0}^{n}\QATOPD[ ] {n%
}{j}_{q}\left( -a\right) ^{n-j}q^{\binom{n-j}{2}}\frac{\left(
bcq^{j},bdq^{j},cdq^{j}\right) _{n-j}}{\left( abcdq^{n+j-1}\right) _{n-j}}%
\psi _{j}\left( x|b,c,d,q\right) \\
=\sum_{j=0}^{n}\QATOPD[ ] {n}{j}_{q}\left( -a\right) ^{n-j}q^{\binom{n-j}{2}}%
\frac{\left( bcq^{j},bdq^{j},cdq^{j}\right) _{n-j}}{\left(
abcdq^{n+j-1}\right) _{n-j}} \\
\times \sum_{k=0}^{j}\left[ \QATOP{j}{k}\right] _{q}(-b)^{j-k}q^{\binom{j-k}{%
2}}\left( cdq^{k}\right) _{j-k}Q_{k}\left( x|c,d,q\right) \\
=\sum_{k=0}^{n}\left[ \QATOP{n}{k}\right] _{q}\left( -1\right) ^{n-k}\left(
cdq^{k}\right) _{n-k}Q_{k}\left( x|c,d,q\right) \times \\
\sum_{m=0}^{n-k}\left[ \QATOP{n-k}{m}\right] _{q}q^{\binom{n-k-m}{2}+\binom{m%
}{2}}a^{n-k-m}b^{m}\frac{\left( bcq^{k+m},bdq^{k+m}\right) _{n-k-m}}{\left(
abcdq^{n+k+m-1}\right) _{n-k-m}}.
\end{gather*}%
Now we notice that $(n-m)(n-m-1)/2+m(m-1)/2-n(n-1)/2\allowbreak =\allowbreak
m\left( m-n\right) .$

Hence we see that $q^{\binom{n-k-m}{2}+\binom{m}{2}}\allowbreak =\allowbreak
q^{\binom{n-k}{2}}q^{m(m-n+k)}$ and consequently we get (\ref{aw_na_asc}).
\end{proof}

\begin{proof}[Proof of Lemma \protect\ref{conversion}]
Proof is based on two formulae obtained independently expressing the sum $%
\sum_{k\geq 0}h_{n+k}\left( x|q\right) h_{m+k}\left( y|q\right) \frac{t^{k}}{%
\left( q\right) _{k}}.$ Now recall that Carlitz in (\cite{Carlitz72}, (1.4))
summed up the following expression $\sum_{k\geq 0}R_{m+k}\left( x|q\right)
R_{n+k}\left( y|q\right) \frac{\tau ^{k}}{\left( q\right) _{k}},$ where $%
R_{n}\left( x|q\right) \allowbreak $ denotes the so called Rogers-Szeg\"{o}
(RS) polynomial defined by :$R_{n}\left( x|q\right) \allowbreak =\allowbreak
\sum_{k=0}^{n}\QATOPD[ ] {n}{k}_{q}x^{k}$. Let us also recall formula $%
h_{n}(\cos \theta |q)\allowbreak =\allowbreak \exp (ni\theta
)R_{n}(e^{-2i\theta }|q)$ (see \cite{IA}, (13.1.7)) expressing relationship
between RS and qH polynomials.

Applying (\cite{Carlitz72}, (1.4)) with $a\allowbreak =\allowbreak \exp
\left( -2i\theta \right) ,$ $b\allowbreak =\allowbreak \exp \left( -2i\eta
\right) $ and $z\allowbreak =\exp \left( i(\theta +\eta \right)
t)\allowbreak $ we get: 
\begin{gather*}
\sum_{k\geq 0}R_{m+k}\left( e^{-2i\theta }|q\right) R_{n+k}\left( e^{-2i\eta
}|q\right) \frac{e^{ik\left( \theta +\eta \right) }t^{k}}{\left( q\right)
_{k}}\allowbreak = \\
\allowbreak \sum_{k\geq 0}R_{k}\left( e^{-2i\theta }|q\right) R_{k}\left(
e^{-2i\eta }|q\right) \frac{e^{ik\left( \theta +\eta \right) }t^{k}}{\left(
q\right) _{k}}\times \\
\sum_{s=0}^{m}\sum_{r=0}^{n}\left[ \QATOP{m}{s}\right] _{q}\left[ \QATOP{n}{r%
}\right] _{q}\frac{\left( te^{i(-\theta +\eta )}\right) _{s}\left(
te^{i(\theta -\eta )}\right) _{r}\left( te^{-i(\theta +\eta )}\right) _{s+r}%
}{\left( t^{2}\right) _{s+r}}e^{-2i\left( m-s\right) \theta }e^{-2i(n-r)\eta
}.
\end{gather*}%
Multiplying both sides by $e^{i(m\theta +n\eta )}$ and passing to
polynomials $h_{n}$ we get%
\begin{gather}
\sum_{k\geq 0}h_{m+k}\left( x|q\right) h_{n+k}\left( y|q\right) \frac{t^{k}}{%
\left( q\right) _{k}}\allowbreak =\allowbreak \sum_{k\geq 0}h_{k}\left(
x|q\right) h_{k}\left( y|q\right) \frac{t^{k}}{\left( q\right) _{k}}
\label{carh} \\
\times \sum_{s=0}^{m}\sum_{r=0}^{n}\left[ \QATOP{m}{s}\right] _{q}\left[ 
\QATOP{n}{r}\right] _{q}\frac{\left( te^{i(-\theta +\eta )}\right)
_{s}\left( te^{i(\theta -\eta )}\right) _{r}\left( te^{-i(\theta +\eta
)}\right) _{s+r}}{\left( t^{2}\right) _{s+r}}e^{i\theta (2s-m)}e^{i\eta
\left( 2r-n\right) }.  \notag
\end{gather}%
Now recall that in \cite{Szabl-JFA}, Lemma 3 i) we have proved certain
formula for the polynomials $H_{n}$ related to polynomials $h_{n}$ as in the
beginning of Section \ref{prob}. After necessary adjustment to polynomials $%
h_{n}$ (\ref{carh}) we get our assertion.
\end{proof}

\begin{proof}[Proof of Theorem \protect\ref{main}]
First we will prove (\ref{ex_p_na_w}). We start with (\ref{asw_na_aw}),
apply (\ref{newparam}) getting%
\begin{gather*}
p_{n}\left( x|y,\rho _{1},q\right) \allowbreak = \\
\allowbreak \sum_{j=0}^{n}\left[ \QATOP{n}{j}\right] _{q}\left( \rho
_{1}^{2}q^{j}\right) _{n-j}\rho _{2}^{n-j}w_{j}\left( x|y,\rho _{1},z,\rho
_{2},q\right) \times \\
\sum_{m=0}^{n-j}\left[ \QATOP{n-j}{m}\right] _{q}e^{-i\left( n-j-m\right)
\eta }e^{im\eta }\frac{\left( \rho _{1}\rho _{2}e^{i\left( \eta -\theta
\right) }q^{j},\rho _{1}\rho _{2}e^{-i\left( \eta +\theta \right)
}q^{j}\right) _{m}}{\left( \rho _{1}^{2}\rho _{2}^{2}q^{2j}\right) _{m}}.
\end{gather*}%
Now we apply (\ref{wzor}) with $m\allowbreak \longrightarrow \allowbreak 0$, 
$n\longrightarrow n-j,$ $t\longrightarrow \rho _{1}\rho _{2}q^{j}$ and get 
\begin{gather*}
\sum_{k=0}^{n-j}\left[ \QATOP{n-j}{k}\right] _{q}e^{-i\left( n-j-k\right)
\eta }e^{ik\eta }\frac{\left( \rho _{1}\rho _{2}e^{i\left( \theta -\eta
\right) }q^{j},\rho _{1}\rho _{2}e^{-i\left( \theta +\eta \right)
}q^{j}\right) _{k}}{\left( \rho _{1}^{2}\rho _{2}^{2}q^{2j}\right) _{k}}%
\allowbreak = \\
\allowbreak p_{n-j}\left( z|y,\rho _{1}\rho _{2}q^{j},q\right) /\left( \rho
_{1}^{2}\rho _{2}^{2}q^{2j}\right) _{n-j}.
\end{gather*}

Now let us concentrate on proving (\ref{ex_w_na_p}). Using (\ref{aw_na_asc})
and (\ref{newparam}) we get:%
\begin{gather*}
w_{n}\left( x|y,\rho _{1},z,\rho _{2},q\right) \allowbreak =\allowbreak
\sum_{k=0}^{n}\left[ \QATOP{n}{k}\right] _{q}\left( -1\right) ^{n-k}q^{%
\binom{n-k}{2}}\left( \rho _{1}^{2}q^{k}\right) _{n-k}p_{k}\left( x|y,\rho
_{1},q\right) \\
\times \sum_{m=0}^{n-k}\left[ \QATOP{n-k}{m}\right] _{q}q^{m(m-n+k)}\rho
_{2}^{m}e^{im\eta }\rho _{2}^{n-k-m}e^{-i(n-k-m)\eta } \\
\times \frac{\left( q^{n-m}\rho _{1}\rho _{2}e^{i(-\eta +\theta
)},q^{n-m}\rho _{1}\rho _{2}e^{i\left( -\eta -\theta \right) }\right) _{m}}{%
\left( \rho _{1}^{2}\rho _{2}^{2}q^{2n-m-1}\right) _{m}}.
\end{gather*}%
Now we apply formula (\ref{wzor}) with $m\allowbreak \longrightarrow
\allowbreak 0,$ $k\longrightarrow m$, $n\longrightarrow n-k,$ $%
q\longrightarrow q^{-1}$, $t\allowbreak =\allowbreak \rho _{1}\rho
_{2}q^{n-1}.$ We have applying (\ref{qbinq-1}) on the way 
\begin{gather*}
w_{n}\left( x|y,\rho _{1},z,\rho _{2},q\right) =\sum_{k=0}^{n}\left[ \QATOP{n%
}{k}\right] _{q}\left( -1\right) ^{n-k}q^{\binom{n-k}{2}}\left( \rho
_{1}^{2}q^{k}\right) _{n-k}\rho _{2}^{n-k}p_{k}\left( x|y,\rho _{1},q\right)
\\
\times \sum_{m=0}^{n-k}\left[ \QATOP{n-k}{m}\right] _{q^{-1}}e^{-i\left(
n-k-2m\right) \eta }\frac{\left( q^{-m+1}te^{i(-\eta +\theta
)},q^{-m+1}te^{i\left( -\eta -\theta \right) }\right) _{m}}{\left(
t^{2}q^{-m+1}\right) _{m}}\allowbreak \\
=\sum_{k=0}^{n}\left[ \QATOP{n}{k}\right] _{q}\left( -1\right) ^{n-k}q^{%
\binom{n-k}{2}}\left( \rho _{1}^{2}q^{k}\right) _{n-k}\rho
_{2}^{n-k}p_{k}\left( x|y,\rho _{1},q\right) \frac{p_{n-k}\left(
z|y,t,q^{-1}\right) }{\left( t^{2};q^{-1}\right) _{n-k}}.
\end{gather*}%
Now it remains to apply Proposition \ref{odwr_p}, i).and once more (\ref%
{qbinq-1}).
\end{proof}

\begin{proof}[Proof of the Corollary \protect\ref{identies}]
i) We combine (\ref{ex_w_na_p}) and (\ref{ex_p_na_w}) getting:%
\begin{gather*}
w_{n}\left( x|y,\rho _{1},z,\rho _{2},q\right) =\sum_{j=0}^{n}\left[ \QATOP{n%
}{j}\right] _{q}\frac{\rho _{2}^{n-j}\left( \rho _{1}^{2}q^{j}\right) _{n-j}%
}{\left( \rho _{1}^{2}\rho _{2}^{2}q^{n+j-1}\right) _{n-j}}g_{n-j}\left(
z|y,\rho _{1}\rho _{2}q^{n-1},q\right) \\
\times \sum_{k=0}^{j}\left[ \QATOP{j}{k}\right] _{q}w_{k}\left( x|y,\rho
_{1},z,\rho _{2},q\right) \frac{\rho _{2}^{j-k}\left( \rho
_{1}^{2}q^{k}\right) _{j-k}}{\left( \rho _{1}^{2}\rho _{2}^{2}q^{2k}\right)
_{j-k}}p_{j-k}\left( z|y,\rho _{1}\rho _{2}q^{k},q\right) .
\end{gather*}%
Changing the order of summation leads: 
\begin{gather*}
w_{n}\left( x|y,\rho _{1},z,\rho _{2},q\right) =\sum_{k=0}^{n}\QATOPD[ ] {n}{%
k}_{q}w_{k}\left( x|y,\rho _{1},z,\rho _{2},q\right) \rho _{2}^{n-k}\left(
\rho _{1}^{2}q^{k}\right) _{n-k}\times \\
\sum_{j=k}^{n}\QATOPD[ ] {n-k}{j-k}_{q}\frac{g_{n-j}\left( z|y,\rho _{1}\rho
_{2}q^{n-1},q\right) }{\left( \rho _{1}^{2}\rho _{2}^{2}q^{n+j-1}\right)
_{n-j}}\frac{p_{j-k}\left( z|y,\rho _{1}\rho _{2}q^{k},q\right) }{\left(
\rho _{1}^{2}\rho _{2}^{2}q^{2k}\right) _{j-k}} \\
=\sum_{k=0}^{n}\QATOPD[ ] {n}{k}_{q}w_{k}\left( x|y,\rho _{1},z,\rho
_{2},q\right) \rho _{2}^{n-k}\left( \rho _{1}^{2}q^{k}\right) _{n-k}\times \\
\sum_{m=0}^{n-k}\QATOPD[ ] {n-k}{m}_{q}\frac{g_{n-k-m}\left( z|y,\rho
_{1}\rho _{2}q^{n-1},q\right) }{\left( \rho _{1}^{2}\rho
_{2}^{2}q^{n+k+m-1}\right) _{n-k-m}}\frac{p_{m}\left( z|y,\rho _{1}\rho
_{2}q^{k},q\right) }{\left( \rho _{1}^{2}\rho _{2}^{2}q^{2k}\right) _{m}}.
\end{gather*}%
Since expansions in AW polynomials is unique we deduce that \newline
$\sum_{m=0}^{n-k}\QATOPD[ ] {n-k}{m}_{q}\frac{g_{n-k-m}\left( z|y,\rho
_{1}\rho _{2}q^{n-1},q\right) }{\left( \rho _{1}^{2}\rho
_{2}^{2}q^{n+k+m-1}\right) _{n-k-m}}\frac{p_{m}\left( z|y,\rho _{1}\rho
_{2}q^{k},q\right) }{\left( \rho _{1}^{2}\rho _{2}^{2}q^{2k}\right) _{m}}%
\allowbreak =\allowbreak 0$ for $k\allowbreak =\allowbreak 0,\ldots ,n-1.$
Now we replace $\rho _{1}\rho _{2}$ by $t.$ Since we deal with rational
functions we can extend range of all variables to all reals.

ii) We combine (\ref{ex_p_na_w}) and (\ref{ex_w_na_p}) getting 
\begin{gather*}
p_{n}\left( x|y,\rho _{1},q\right) =\sum_{j=0}^{n}\left[ \QATOP{n}{j}\right]
_{q}\frac{\rho _{2}^{n-j}\left( \rho _{1}^{2}q^{j}\right) _{n-j}}{\left(
\rho _{1}^{2}\rho _{2}^{2}q^{2j}\right) _{n-j}}p_{n-j}\left( z|y,\rho
_{1}\rho _{2}q^{j},q\right) \\
\times \sum_{k=0}^{j}\QATOPD[ ] {j}{k}_{q}p_{k}\left( x|y,\rho _{1},q\right) 
\frac{\rho _{2}^{j-k}\left( \rho _{1}^{2}q^{k}\right) _{j-k}}{\left( \rho
_{1}^{2}\rho _{2}^{2}q^{j+k-1}\right) _{j-k}}g_{j-k}\left( z|y,\rho _{1}\rho
_{2}q^{j-1},q\right) .
\end{gather*}

We change the order of summation getting:%
\begin{gather*}
p_{n}\left( x|y,\rho _{1},q\right) =\sum_{k=0}^{n}\QATOPD[ ] {n}{k}%
_{q}p_{k}\left( x|y,\rho _{1},q\right) \rho _{2}^{n-k}\left( \rho
_{1}^{2}q^{k}\right) _{n-k}\times \\
\sum_{j=k}^{n}\QATOPD[ ] {n-k}{j-k}_{q}\frac{p_{n-j}\left( z|y,\rho _{1}\rho
_{2}q^{j},q\right) }{\left( \rho _{1}^{2}\rho _{2}^{2}q^{2j}\right) _{n-j}}%
\frac{g_{j-k}\left( z|y,\rho _{1}\rho _{2}q^{j-1},q\right) }{\left( \rho
_{1}^{2}\rho _{2}^{2}q^{j+k-1}\right) _{j-k}} \\
=\sum_{k=0}^{n}\QATOPD[ ] {n}{k}_{q}p_{k}\left( x|y,\rho _{1},q\right) \rho
_{2}^{n-k}\left( \rho _{1}^{2}q^{k}\right) _{n-k}\times \\
\sum_{m=0}^{n-k}\QATOPD[ ] {n-k}{m}_{q}\frac{p_{n-k-m}\left( z|y,\rho
_{1}\rho _{2}q^{m+k},q\right) }{\left( \rho _{1}^{2}\rho
_{2}^{2}q^{2k+2m}\right) _{n-k-m}}\frac{g_{m}\left( z|y,\rho _{1}\rho
_{2}q^{m+k-1},q\right) }{\left( \rho _{1}^{2}\rho _{2}^{2}q^{2k+m-1}\right)
_{m}}.
\end{gather*}

Since expansion in ASC polynomials is unique we deduce that\newline
$\sum_{m=0}^{n-k}\QATOPD[ ] {n-k}{m}_{q}\frac{p_{n-k-m}\left( z|y,\rho
_{1}\rho _{2}q^{m+k},q\right) }{\left( \rho _{1}^{2}\rho
_{2}^{2}q^{2k+2m}\right) _{n-k-m}}\frac{g_{m}\left( z|y,\rho _{1}\rho
_{2}q^{m+k-1},q\right) }{\left( \rho _{1}^{2}\rho _{2}^{2}q^{2k+m-1}\right)
_{m}}\allowbreak =\allowbreak 0$ for $k=0,\ldots ,n-1.$ Now it remains to
replace $\rho _{1}\rho _{2}$ by $t.$
\end{proof}

\end{document}